\documentclass[12pt]{amsart}
\usepackage{amssymb}
\usepackage{latexsym}
\usepackage{amsmath}
\usepackage{bm}

\def\wt{\widetilde}
\def\wh{\widehat}
\def\ov{\overline}
\def \im{{\rm Im\,}}
 \def\up{\upharpoonright}

\def\cA{\mathcal A} \def\cH{\mathcal H}  \def\cB{\mathcal B} \def\cC{\mathcal C}
\def\cD{\mathcal D} \def\cE {\mathcal E}\def\cF {\mathcal F}   \def\cJ {\mathcal J}

\def\cK{\mathcal K} \def\cL{\mathcal L}
 \def\cN{\mathcal N} \def\cP{\mathcal P}
  \def\cS{\mathcal S}
\def\cU{\mathcal U} 
 \def\cT{\mathcal T} \def\cI{\mathcal I}

\def\B{\mbox{\boldmath$B$}}

\def \gH{\mathfrak H}   \def \gN{\mathfrak N}

\def \bC{\mathbb C} \def \bN{\mathbb N} 
\def\bR{\mathbb R}
 
\def \l{\lambda}
\def \a{\alpha} \def \b{\beta}   \def \L{\Lambda}  \def \s{\sigma} \def \t{\theta}
 
\def\d {\delta}

\def \f{\varphi} \def\D {\Delta} \def\Si{\Sigma}
 \def \G{\Gamma} 
\def \C{\widetilde {\mathcal C}}
 \def \CAt {C(\wt A_\tau)}

\def \cd {\cdot}

\def\lI {\cL_\Delta^2(\cI;\bC^n)} \def\LI {L_\Delta^2(\cI;\bC^n)}
\def\lIW {\cL_\Delta^2(\cI;\bC^m)} \def\LIW {L_\D^2(\cI;\bC^m)}

\def\AC {AC(\cI, \bC^n)} \def\ACm {AC(\cI, \bC^m)}

\def\lS {\cL^2(\s; \bC^p)} \def\LS {L^2(\s ; \bC^p)}
\def\lSm {\cL^2(\s; (\bC^m)^r)} 

\def\tma{\cT_{\max}} \def\tmi{\cT_{\min}} \def\Tma{T_{\max}} \def\Tmi{T_{\min}}

\def\sma{\cS_{\max}}  \def\Sma{S_{\max}} \def\Smi{S_{\min}}

 \def\Tma{T_{\max}}

 \def \cex {\overline {\rm ext }(A)}
\def\CAt{C(\wt A_\tau)}

\def \dom {{\rm dom}\,}  \def \ran {{\rm ran}\,}  \def \ker{{\rm
ker\,}}
 \def \mul {{\rm mul}\,}

\def  \RH {\wt R (\cH)}  

\def \CR {\bC\setminus\bR}

\def\bt{\{\cH,\G_0,\G_1\}}

\newcommand{\Romannumeral}[1]{\uppercase\expandafter{\romannumeral #1\relax}}

\newtheorem{theorem}{Theorem}[section]
\newtheorem{proposition}[theorem]{Proposition}
\newtheorem{corollary}[theorem]{Corollary}
\newtheorem{lemma}[theorem]{Lemma}
\newtheorem{assertion}[theorem]{Assertion}
\theoremstyle{definition}

\theoremstyle{definition}
\newtheorem {definition} [theorem]{Definition}
\theoremstyle{remark}
\newtheorem{remark}[theorem]{Remark}
\numberwithin{equation}{section}
\begin{document}
\title[On uniform convergence of the  inverse Fourier transform]
{On uniform convergence of the inverse Fourier transform for differential equations and Hamiltonian systems with degenerating  weight}
\author{Vadim  Mogilevskii}
\address{Department of Mathematical Analysis and Informatics, Poltava National V.G. Korolenko Pedagogical University,  Ostrogradski Str. 2, 36000 Poltava, Ukraine }
\email{vadim.mogilevskii@gmail.com}

%\date{DD/MM/2004}
\begin{abstract}
We study pseudospectral and spectral functions for Hamiltonian system $Jy'-B(t)=\l\D(t)y$  and differential equation $l[y]=\l\D(t)y$  with matrix-valued coefficients defined on an interval $\cI=[a,b)$ with the regular endpoint $a$. It is not assumed that the matrix weight $\D(t)\geq 0$ is invertible a.e. on $\cI$. In this case a pseudospectral function always exists, but the set of spectral functions may be empty. We obtain a parametrization $\s=\s_\tau$ of all pseudospectral and  spectral functions $\s$ by means of a Nevanlinna  parameter $\tau$ and single out in terms of $\tau$  and boundary conditions the class  of functions $y$ for which the inverse Fourier transform $y(t)=\int\limits_\bR \f(t,s)\, d\s (s) \wh y(s)$ converges uniformly. We also show that  for scalar equation $l[y]=\l \D(t)y$ the set of spectral functions is not empty. This enables us to extend the Kats-Krein and Atkinson results for scalar Sturm - Liouville equation $-(p(t)y')'+q(t)y=\l \D (t) y$ to such equations with arbitrary coefficients $p(t)$ and $q(t)$ and arbitrary non trivial weight $\D (t)\geq 0$.
\end{abstract}
\maketitle
\section{Introduction}
We consider the differential equation of an even order $2r$
\begin{gather}\label{1.1}
l[y]= \sum_{k=0}^r  (-1)^k \left( p_{r-k}(t)y^{(k)}\right)^{(k)}=\l \D(t)y,\quad t\in \cI=[a,b\rangle , \quad -\infty <a<b\leq \infty
\end{gather}
and its natural generalization -- the Hamiltonian differential system
\begin{gather}\label{1.2}
Jy'-B(t)y= \l \D(t)y, \quad t\in \cI=[a,b\rangle , \quad -\infty <a<b\leq \infty
\end{gather}
on an interval $\cI=[a,b\rangle $ with the regular endpoint $a$ and arbitrary (regular or singular) endpoint $b$. It is assumed that the coefficients $p_j$ and the weight $\D$ in \eqref{1.1} are functions on $\cI$ with values in the set  $\B (\bC^m)$  of all linear operators in $\bC^m$ (or equivalently  $m\times m$-matrices) such that $p_j=p_j^*,\; \D\geq 0 $ (a.e. on $\cI$) and $p_0^{-1}, \; p_1, \dots, p_r, \D$ are locally integrable. As to system \eqref{1.2}, we assume that $J\in \B(\bC^n)\; (n=2p)$ is given by
\begin {equation} \label{1.3}
J=\begin{pmatrix} 0 & -I_p \cr  I_p&
0\end{pmatrix}:\underbrace{\bC^p\oplus \bC^p}_{\bC^n} \to \underbrace{\bC^p\oplus \bC^p}_{\bC^n}
\end{equation}
and $B$ and $\D$ are locally integrable $\B (\bC^n)$-valued functions on $\cI$ such that $B=B^*$ and $\D\geq 0$ a.e. on $\cI$. Equation  \eqref{1.1} (system \eqref{1.2}) is called regular if $b<\infty$ and $p_0^{-1}, \; p_1, \dots, p_r, \D$ (resp. $B,\D$)  are integrable on $\cI$; otherwise it is  called singular. Equation \eqref{1.1} is called scalar if $m=1$ and hence $p_j$ and $\D$ are real valued functions.

Following to \cite{BinVol13} we  call the weight $\D$ definite if it is invertible a.e. on $\cI$ and semi-definite in the opposite case. Moreover, the weight $\D$ in the scalar equation \eqref{1.1} is called nontrivial if the equality $\D(t)=0$ (a.e. on $\cI$) does not hold. Clearly, non triviality  is the weakest restriction on $\D$, which saves the interest to studying of \eqref{1.1}.

As is known a spectral function is a fundamental concept in the spectral theory of differential equations \cite{DunSch,Nai,Sht57,Wei} and Hamiltonian systems \cite{AD12,Kac03,Sah13}. Let $\f(\cd,\l)(\in \B(\bC^p,\bC^p\oplus\bC^p))$ be an operator solution of \eqref{1.2} such that $\f(a,\l)=(-\sin A, \cos A)^\top$ with some $A=A^*\in \B(\bC^p)$. Then a spectral function of the system \eqref{1.2} is defined as an operator-valued (or, equivalently, matrix-valued) distribution function $\s(s)(\in \B(\bC^p))$ such that the generalized Fourier  transform
\begin {equation}\label{1.4}
L_\D^2(\cI)\ni f(t)\to \wh f(s)=\int_\cI \f^*(t,s)\D(t)f(t)\,dt
\end{equation}
induces an isometry $V_\s$ from the Hilbert space $L_\D^2(\cI)$ of all vector-functions $f(t)(\in\bC^n)$ such that $\int\limits_{\cI}(\D (t)f(t),f(t))\, dt <\infty$ to the Hilbert space $L^2(\s;\bC^p)$. Similarly one defines a spectral function $\s(s)(\in \B((\bC^m)^r))$ of equation \eqref{1.1}. If $\s(\cd)$ is a spectral function of \eqref{1.1} or \eqref{1.2}, then for each $y\in L_\D^2(\cI)$ the inverse Fourier transform is
\begin{gather}\label{1.5}
y(t)=\int_\bR \f(t,s)\, d\s (s) \wh y(s),
\end{gather}
where the integral converges in $L_\D^2(\cI)$. Recall also that a spectral function $\s(\cd)$ is called orthogonal if $V_\s$ is a unitary operator.

Existence of a spectral function for equation \eqref{1.1} and system \eqref{1.2} with the definite weight is a classical result (see e.g. \cite{Wei}). This result was extended by I.S. Kats \cite{Kac69,Kac71} to the scalar Sturm-Liouville equation
\begin{gather}\label{1.7}
l[y]=-(p(t)y')' + q(t)y=\l \D(t)y, \quad t\in\cI=[a,b\rangle, \quad \l\in\bC
\end{gather}
with $p(t)\equiv 1$ and the semi-definite weight $\D$. Moreover, I.S. Kats and M.G. Krein parameterized in \cite[\S 14]{KacKre} all spectral functions of such an equation under the following additional conditions:

(A1) there is no  interval $(a,b')\subset\cI$ ($(a',b)\subset \cI$) such that $\D(t)=0$ a.e. on  $(a,b')$ (resp. on $(a',b)$);

(A2) if $\D(t)=0$ a.e. on an interval $(a',b')\subset \cI$, then $q(t)=0 $ (a.e. on $(a',b')$).

The Kats -- Krein parametrization  can be formulated as the following theorem.
\begin{theorem}\label{th1.0}
Consider scalar regular  equation \eqref{1.7}  such that  $p(t)\equiv 1$ and {\rm (A1)} and {\rm (A2)} are satisfied. Let   $\f(\cd,\l)$ and $\psi(\cd,\l)$ be solutions of \eqref{1.7} with
\begin{gather}\label{1.8}
\f(a,\l)=-\sin\a, \quad \f'(a,\l)=\cos\a,\;\;
\psi(a,\l)=-\cos\a, \quad \psi'(a,\l)=-\sin\a
\end{gather}
and let $\wh R[\bC]=R[\bC]\cup \{\tau(\l)\equiv\infty\}$, where $R[\bC]$ is the class of all  complex-valued Nevanlinna functions $\tau(\l)$ (see Section \ref{sect2.1}).  Then  the equalities
\begin{gather}
m_\tau(\l)=\frac {\psi (b,\l)\tau(\l)-\psi' (b,\l)}{
\f (b,\l)\tau(\l)-\f'(b,\l) }, \quad \l\in\CR\label{1.10}\\
\s_\tau(s)=\lim\limits_{\d\to+0}\lim\limits_{\varepsilon\to +0} \frac 1 \pi
\int_{-\d}^{s-\d}\im \,m_\tau(u+i\varepsilon)\, du \label{1.11}
\end{gather}
establish a bijective correspondence $\s(\cd)=\s_\tau(\cd)$ between all functions $\tau\in \wh R[\bC]$ and all (real valued) spectral functions $\s(\cd)$ of \eqref{1.7} (with respect to the Fourier transform \eqref{1.4}). Moreover, $\s_\tau(\cd)$ is orthogonal  if and only if $\tau(\l)\equiv\t(=\ov\t)$ or $\tau(\l)\equiv\infty,\;\l\in\CR$.
\end{theorem}
As is known each orthogonal spectral function $\s(\cd)$ of the equation \eqref{1.1} with definite weight is associated with a certain self-adjoint operator $\wt S$ in $L_\D^2(\cI)$. Moreover, a classical result claims that for each function $y$ from the domain of $\wt S$ the $L_\D^2(\cI)$-convergence in \eqref{1.5} can be improved to uniform convergence on each compact interval $[a,c]\subset\cI$ (see e.g. \cite[Theorem \Romannumeral{13}.5.16]{DunSch}. In the case of the Sturm -- Liouville equation this result yields the following theorem (see e.g. \cite{Col}).
\begin{theorem}\label{th1.1}
Consider the eigenvalue problem for scalar regular Sturm-Liouville equation \eqref{1.7} with the definite weight $\D$ subject to self-adjoint boundary conditions
\begin{gather}\label{1.13}
\cos\a\cd y(a) + \sin \a \cd (py')(a)=0, \qquad \cos\b\cd y(b) + \sin \b \cd (py')(b)=0.
\end{gather}
Then each function $y\in AC(\cI)$ such that $py'\in AC(\cI)$, $\D^{-1}l[y]\in L_\D^2(\cI)$ and \eqref{1.13} is satisfied admits the eigenfunction expansion
\begin{gather}\label{1.14}
y(t)=\sum_{k=1}^\infty (y,v_k)_\D v_k(t), \quad t\in\cI,
\end{gather}
which converges absolutely and uniformly on $\cI$. In \eqref{1.14} $\{v_k\}_1^\infty$ are orthonormal eigenfunctions of the problem \eqref{1.7}, \eqref{1.13}.
\end{theorem}
F. Atkinson in \cite[Theorem 8.9.1]{Atk} extended Theorem \ref{th1.1} to scalar regular equations  \eqref{1.7} with semi-definite  weight $\D$ satisfying the condition $0\leq p(t)\leq \infty, \; t\in\cI,$ and assumptions (A1) and (A2) before Theorem \ref{th1.0}. Moreover, Theorem \ref{th1.1} was extended to eigenvalue problems for regular scalar equations \eqref{1.7} \cite{Ful77,Hin79} and \eqref{1.1} \cite{Bin02} with definite weight subject to boundary conditions linearly dependent on the eigenparameter $\l$. It is worth to note that these papers deal in fact with a special class of nonorthogonal spectral functions. Observe also that various properties (existence and behavior of eigenvalues, oscillation of eigenfunctions etc.) of eigenvalue problems for Sturm -- Liouville equations with semi-definite weight was studied in \cite{BinVol13}.

It turns out that a spectral function of the system \eqref{1.2} and equation \eqref{1.1} with semi-definite weight may not exist and hence definition of a spectral function requires a certain modification. To this end one defines a pseudospectral function of the system \eqref{1.2} as an operator-valued distribution function $\s(s)(\in \B(\bC^p))$ such that the generalized Fourier transform \eqref{1.4} induces a partial isometry $V_\s:L_\D^2(\cI)\to L^2(\s;\bC^p)$ with the minimally possible kernel $\ker V_\s$ (see \cite{Kac03,AD12,Sah13} for regular systems and \cite{Mog15} for singular ones). If $\s(\cd)$ is a pseudospectral
function, then the inverse Fourier transform \eqref{1.5} holds only for functions $y\in L_\D^2(\cI)\ominus \ker V_\s$. It turns out that a pseudospectral function exists for any system \eqref{1.2}; moreover, either the set of  spectral functions of a given system is empty or it coincides with the set of pseudospectral ones. The Kats -- Krein parametrization of  spectral functions   was extended in \cite{AD12,Mog15,Sah13} to Hamiltonian systems \eqref{1.2}. In these papers a parametrization $\s(\cd)=\s_\tau(\cd)$ of pseudospectral functions $\s(\cd)$ is given in terms of the parameter $\tau=\tau(\l)$,  which takes on values in the set of all relation-valued Nevanlinna functions (for more details see Theorem \ref{th3.12}).

In the present paper we extend the above results concerning the uniform convergence of the inverse Fourier transform \eqref{1.5} to arbitrary (possibly nonorthogonal) pseudospectral and spectral functions of differential equation \eqref{1.1} and Hamiltonian system \eqref{1.2} with matrix-valued coefficients and semi-definite weight $\D$. This enables us to extend Theorems \ref{th1.0} and \ref{th1.1} to scalar regular Sturm - Liouville equation \eqref{1.7} with arbitrary coefficients $p$ and $q$ and semi-definite nontrivial weight  $\D$.

First we consider Hamiltonian system \eqref{1.2}. Assume for simplicity that the set of spectral functions of this system is not empty. Let $\tau=\tau(\l)$ be a Nevanlinna parameter and let $\s(\cd)=\s_\tau(\cd)$ be the corresponding spectral function of the system. We prove the following statement:

(S) If $y\in L_\D^2(\cI)$ is an absolutely continuous vector-function such that the equality $Jy'-By=\D f_y$ holds with some $f_y\in L_\D^2(\cI)$ and the boundary conditions
\begin{gather}\label{1.16}
(\cos A,\, \sin A)\, y(a)=0, \qquad \G_b y\in\eta_\tau
\end{gather}
are satisfied, then the inverse Fourier transform \eqref{1.5} converges absolutely and uniformly on each compact interval $[a,c]\in\cI$. In \eqref{1.16} $A=A^*\in \B(\bC^p)$, $\G_b y$ is a singular boundary value of $y$ at the endpoint $b$ (in the case of the regular system one can put $\G_b y = y(b)$) and $\eta_\tau$ is a linear relation defined in terms of the asymptotic behavior   of the parameter $\tau(\l)$ at the infinity.

If $\tau(\l)\equiv \t$ is a self-adjoint parameter, then the spectral function $\s_\tau(\cd)$ is orthogonal, $\eta_\tau=\t$ and  \eqref{1.16} turns into self-adjoint boundary conditions, which defines  a self-adjoint operator $\wt T$ in $L_\D^2(\cI)$. So in this case under the additional assumption of definiteness  of $\D$ statement (S) gives rise to  known results on the uniform convergence \cite{DunSch}. Note also that in fact we prove  statement (S) for pseudospectral functions (see Theorem \ref{th3.16}).

As is known \cite{KogRof75} equation \eqref{1.1} is equivalent to a certain special system \eqref{1.2}. Therefore the concept of a pseudospectral function and  relative results can be readily transformed to equation \eqref{1.1} with matrix-valued coefficients and semi-definite weight (see Theorems \ref{th4.8} and \ref{th4.11}). Nevertheless it turns out that scalar equation \eqref{1.1} with semi-definite  nontrivial weight possesses an essential peculiarity. Namely, we show (see Theorem \ref{th4.14}) that the set of spectral functions of such an equation is not empty. Moreover, we parameterize all these spectral functions by means of a Nevanlinna  parameter $\tau$ and single out in terms of $\tau$  and boundary conditions the class  of functions $y\in L_\D^2(\cI)$ for which the inverse Fourier transform \eqref{1.5} with the spectral function $\s(\cd)=\s_\tau (\cd)$ converges uniformly on each compact interval $[a,c]\subset\cI$ (see Theorems \ref{th4.14}, \ref{th4.15} and \eqref{th4.18.2}). In the case of the Sturm -- Liouville equation these results can be formulated in the form of the  following theorem.
\begin{theorem}\label{th1.2}
Consider scalar regular equation \eqref{1.7} on $\cI=[a,b]$ with real-valued coefficients $p,q$ and  semi-definite nontrivial weight  $\D (t)\geq 0 \,$ ($p^{-1},q,\D\in L^1(\cI)$). Denote by $\dom l$ the set of all functions $y\in AC(\cI)$ such that $y^{[1]}:=py'\in AC(\cI)$ and let $l[y]:=-(y^{[1]})'+qy,\; y\in\dom l$. Moreover, let  $\f(\cd,\l)\in\dom l$ and $\psi(\cd,\l)\in\dom l$ be solutions of \eqref{1.7} defined by initial values  \eqref{1.8} with $\f^{[1]}(a,\l)$ and $\psi^{[1]} (a,\l)$ instead of $\f'(a,\l)$ and $\psi'(a,\l)$ respectively. Then:

{\rm (i)} The set of spectral functions of \eqref{1.7} (with respect to the Fourier transform \eqref{1.4}) is not empty and statement of Theorem \ref{th1.0} is valid.

{\rm (ii)}Let $\tau=\tau(\cd)\in R[\bC]$ and let $\s(\cd)=\s_\tau(\cd) $ be the corresponding spectral function of \eqref{1.7}  defined by \eqref{1.10} and  \eqref{1.11}. Denote by $\cF$ the set of all functions $y\in \dom l$ satisfying the following conditions: (a) there exists a function $f_y\in\cL_\D^2(\cI)$ such that $l[y]=\D f_y $(a.e. on $\cI$); (b)  one of the following boundary conditions {\rm (bc1)} -- {\rm (bc3)}  dependent on $\tau$ are satisfied:

{\rm (bc1)} if $\lim\limits_{y\to\infty} \frac {\tau(iy)} {iy}\neq 0$, then $\cos\a\cd y(a) + \sin \a \cd y^{[1]}(a)=0\;$ and $\; y(b)=0$;

{\rm (bc2)} if
\begin{gather}\label{1.19}
\lim\limits_{y\to\infty} \tfrac {\tau(iy)} {iy}= 0 \;\;\; {\rm and} \;\;\; \lim_{y\to\infty} y\im \tau(iy)<\infty,
\end{gather}
then $\cos\a\cd y(a) + \sin \a \cd y^{[1]}(a)=0\;$ and $\; y^{[1]}(b)=D_\tau y(b)$ (here $D_\tau=\lim\limits_{y\to\infty}\tau(iy)$);

{\rm (bc3)} if $\lim\limits_{y\to\infty} \frac {\tau(iy)} {iy}= 0$ and $ \lim\limits_{y\to\infty} y\im \tau(iy)=\infty$, then

\centerline{$\cos\a\cd y(a) + \sin \a \cd y^{[1]}(a)=0, \;\; y(b)=0\;\;{\rm and}\;\; y^{[1]}(b)=0$.}

Then for each function $y\in\cF$
\begin{gather}\label{1.22}
y(t)=\int_{\bR} \f(t,s) \wh y(s)\, d\s(s),
\end{gather}
where the integral converges absolutely and uniformly on $\cI$.
\end{theorem}
Note that statement (i) of Theorem \ref{th1.2} extends the Kats existence theorem \cite{Kac69,Kac71} and  Kats -Krein parametrization of spectral functions to Sturm-Liouville equations \eqref{1.7} with $p(t)\not\equiv 1$ and semi-definite nontrivial weight $\D$ (cf Theorem \ref{th1.0}). Moreover, by using Theorem \ref{th1.2} we extend to such  equations  Theorem \ref{th1.1} (see Corollary \ref{cor4.18.5}). In other words, we show that in the case $p(t)<\infty$ Theorem \ref{th1.1} remains valid without Atkinson's assumptions.

In conclusion note that our investigations are based on the results of \cite{Mog19} (see also \cite{DajLan18}), where compression $P_\gH \wt A\up \gH$ of an exit space   extension $\wt A=\wt A^*$ of an operator $A\subset A^*$  in the Hilbert space $\gH$ are characterized in terms of abstract boundary conditions. We show that in the case of a nonorthogonal spectral function $\s(\cd)$ the integral in \eqref{1.5} converges uniformly for any $y$ from the domain of the compression of respective $\wt A$ and then apply the results of \cite{Mog19} to this compression.

\section{Preliminaries}
\subsection{Notations}\label{sect2.1}
The following notations will be used throughout the paper: $\gH$, $\cH$ denote separable  Hilbert spaces; $\B (\cH_1,\cH_2)$  is the set of all bounded linear operators defined on $\cH_1$ with values in  $\cH_2$;  $A\up \mathcal L$ is a restriction of the operator $A\in \B(\cH_1,\cH_2)$ to the linear manifold $\mathcal L\subset\cH_1$; $P_\cL$ is the orthoprojection in $\gH$ onto the subspace $\cL\subset \gH$;  $\bC_+\,(\bC_-)$ is the open  upper (lower) half-plane  of the complex plane; $ \cA$ is the $\s$-algebra of  Borel sets in $\bR$ and $\mu$ is the Borel measure on $\cA$. For a set $B\subset\bR$ we denote by $\chi_B(\cd)$ the indicator of $B$, i.e., the real-valued function on $\bR$ given by $\chi_B(t)=1$ for $t\in B$ and $\chi_B(t)=0$ for $t\in \bR\setminus B$.

Recall that a linear manifold $T$ in the Hilbert space $\cH_0\oplus\cH_1$ ($\cH\oplus\cH$) is called a  linear relation from $\cH_0$ to $\cH_1$ (resp. in $\cH$). The set of all closed linear relations from $\cH_0$ to $\cH_1$ (in $\cH$) will be denoted by $\C (\cH_0,\cH_1)$ (resp. $\C(\cH)$). Clearly for each linear operator $T:\dom T\to\cH_1, \;\dom T\subset \cH_0,$ its graph ${\rm gr} T =\{\{f,Tf\}:f\in \dom T\} $  is a linear relation from $\cH_0$ to $\cH_1$. This fact enables one to consider an operator $T$ as a linear relation. In the following we denote by $\cC (\cH_0,\cH_1)$ the set of all closed linear operators $T:\dom T\to\cH_1, \;\dom T\subset \cH_0$. Moreover, we let  $\cC (\cH)=\cC (\cH,\cH)$.

For a linear relation $T$ from $\cH_0$ to $\cH_1$  we denote by $\dom T, \; \ker T,\; \ran T $ and $\mul T:=\{h_1\in\cH_1: \{0,h_1\}\in T\}$
the domain, kernel, range and multivalued part  of
$T$ respectively. Denote also by $T^{-1}$ and $T^*$ the inverse and adjoint linear relations of $T$ respectively. Clearly, $T$ is an operator if and only if $\mul T=\{0\}$.

We will use the following notations:

(i) $R[\cH]$  is the set of all Nevanlinna $\B(\cH)$-valued functions, i.e., the set of all  holomorphic operator functions $M(\cd):\CR\to \B(\cH)$ such that   $\im\l\cd\im M (\l)\geq 0$ and $M^*(\l)=M (\ov\l), \; \l\in\CR$;

(ii) $R_u[\cH]$  is the set of all functions $M(\cd)\in R[\cH]$ such that  $(\im M(\l))^{-1}\in\B (\cH)$ for all $\l\in\CR$;

(iii) $\RH$ is the set of all Nevanlinna relation-valued functions (see e.g. \cite{DM06}), which in the case $\cH=\bC^m$ can be defined as the set of all functions  $\tau(\cd):\CR\to \C (\bC^m)$ such that   $\mul \tau(\l):=\cK$  does not depend on $\l\in\CR$ and the decompositions
\begin{gather}\label{2.9}
\bC^m=\cH_0 \oplus\cK, \qquad \tau (\l)={\rm gr}\,\tau_0(\l)\oplus \wh \cK, \quad \l\in\CR
\end{gather}
hold with $\wh\cK=\{0\}\oplus \cK$ and $\tau_0(\cd)\in R [\cH_0]$ (the operator function $\tau_0(\cd)$ is called the operator part of $\tau(\cd)$).

It is clear that $R[\cH]\subset \wt R(\cH)$.
\subsection{Boundary triplets and compressions of exit space extensions}
Recall that a linear relation $T$ in $\gH$ is called symmetric (self-adjoint) if $T\subset T^*$ (resp. $T=T^*$). In the following we denote by $A$ a closed symmetric linear relation  in a Hilbert space $\gH$. Let  $\gN_\l(A)=\ker (A^*-\l)\; (\l\in\CR)$ be a defect subspace of $A$ and let $n_\pm (A):=\dim \gN_\l(A),\;
\l\in\bC_\pm,$ be deficiency indices of $A$. Denote by $\cex$ the
set of all closed proper extensions of $A$ (i.e., the set of all
relations $\wt A \in\C(\gH)$ such that $A\subset\wt A\subset A^*$).

It is easy to see that $A$ is a densely defined operator if and only if $\mul A^*=\{0\}$.

As is known a linear relation $\wt A=\wt A^*$ in a Hilbert space $\wt\gH\supset \gH$ is called an exit space extension of $A$ if $A\subset \wt A$ and the  minimality condition $\ov{{\rm span}} \{\gH,(\wt A-\l)^{-1}\gH: \l\in\CR\}=\wt\gH$
is satisfied.
\begin{definition}\label{def2.3}$ \,$\cite{GorGor}
A collection $\Pi=\bt$ consisting of a Hilbert space $\cH$ and linear mappings   $\G_j:A^*\to \cH, \; j\in\{0,1\},$  is called a
boundary triplet for $A^*$, if the mapping $\G=(\G_0,\G_1)^\top $ from $A^*$ into
$\cH\oplus\cH$ is surjective and the following abstract Green's
identity  holds:
\begin {equation*}
(f',g)-(f,g')=(\G_1  \wh f,\G_0 \wh g)- (\G_0 \wh f,\G_1
\wh g), \quad \wh f=\{f,f'\}, \; \wh g=\{g,g'\}\in A^*.
\end{equation*}
\end{definition}
\begin{theorem}\label{th2.6}$\,$ \cite{DM91,Mal92}
Let $\Pi=\bt$ be a boundary triplet for  $A^*$. Then:

{\rm (\romannumeral 1)}   The
mapping
\begin {equation}\label{2.11}
\t\to  A_\t :=\{ \hat f\in A^*:\{\G_0  \hat f,\G_1 \hat f \}\in
\t\}
\end{equation}
establishes a bijective correspondence $\wt A=A_\t$ between all linear relations  $\t \in\C(\cH)$ and all extensions $ \wt A= \cex$. Moreover $A_\t$ is symmetric (self-adjoint)  if  and only if $\t$ is symmetric (resp. self-adjoint).

{\rm (\romannumeral 2)} The equality $P_\gH (\wt A_\tau -\l)^{-1}\up\gH =(A_{-\tau(\l)}-\l)^{-1}, \; \l\in\CR,$ gives a bijective correspondence $\wt A=\wt A_\tau$ between all functions $\tau=\tau(\cd)\in \RH$ and all exit space extensions $\wt A=\wt A^*$ of $A$. Moreover, if $\tau(\l)\equiv \t(=\t^*), \; \l\in\CR$, then $\wt A_\tau = A_{-\t}$ (see \eqref{2.11}).
\end{theorem}
Note that the same  parametrization $\wt A=\wt A_\tau$ of exit space extensions $\wt A$ of $A$ can be also given by means of the Krein formula for generalized resolvents (see e.g. \cite{LanTex77,DM91,Mal92}).
\begin{definition}\label{def2.8}
The linear relation $C(\wt A)$ in $\gH$ defined by
\begin{gather*}
C(\wt A):=P_{\gH}\wt A\up \gH=\{\{f,P_\gH f'\}: \{f,f'\}\in\wt A, \; f\in\gH\}
\end{gather*}
is called the compression of the exit space  $\wt A =\wt A^*$ of $A$.
\end{definition}
Clearly, $C(\wt A)$ is a  symmetric extension of $A$. Note also that the equality
\begin{gather}\label{2.12.1}
\Phi(\wt A):=\{\{P_\gH f, P_\gH f'\}:\{f,f'\}\in \wt A \}
\end{gather}
defines a linear relation $\Phi (\wt A)\subset A^*$ (see e.g. \cite{DM06}).

A characterization of the compression $\CAt$ in terms of the parameter $\tau$ is given by the following theorem obtained in our paper \cite{Mog19}.
\begin{theorem}\label{th2.9}
Assume that $\Pi=\{\bC^m,\G_0,\G_1\}$ is a boundary triplet for $A^*$ (in this case $n_+(A)=n_-(A)=m$). Let  $\tau\in \wt R(\bC^m)$, let $\wt A_\tau=\wt A_\tau^*$ be the corresponding exit space extension of $A$ and let $\CAt$ be the compression of $\wt A_\tau$. Assume also that $\tau_0\in R[\cH_0]$  and $\cK$ are the operator and multivalued parts of $\tau$ respectively  (see \eqref{2.9}). Then:

{\rm (i)} the equalities  $\cB_{\tau_0}=\lim\limits_{y\to\infty}\tfrac 1 {iy}\tau_0(iy)$ and
\begin{gather*}
\dom D_{\tau_0} =\{h\in\cH_0: \lim_{y\to\infty}  y\im (\tau_0(iy)h,h)< \infty\}, \quad D_{\tau_0} h=\lim_{y\to\infty}  \tau_0(iy)h, \quad h\in \dom \cN_{\tau_0}
\end{gather*}
correctly define the nonnegative  operator $\cB_{\tau_0}\in \B (\cH_0)$ and the operator  $D_{\tau_0}:\dom D_{\tau_0}\to \cH_0 \;\; (\dom D_{\tau_0} \subset \cH_0) $;

{\rm (i)} $\CAt=A_{\eta_\tau}$  with the symmetric linear relation $\eta_\tau\in \C (\bC^m)$ given by
\begin{gather}\label{2.13}
\eta_\tau=\{\{h,- D_{\tau_0}h+\cB_{\tau_0}h'+k \}: h\in\dom D_{\tau_0}, h'\in  \cH_0, k\in\cK\}.
\end{gather}
\end{theorem}
\subsection{The spaces $\cL^2(\s;\bC^m)$ and $L^2(\s;\bC^m)$ }
Recall that a non-decreasing operator function $\s(\cd): \bR\to \B(\bC^m)$ is called a
distribution function if it is left continuous and satisfies $\s(0)=0$.
\begin{theorem}\label{th2.10} $\,$\cite{DunSch,MalMal03}
Let $\s(\cd): \bR\to \B (\bC^m)$  be a distribution function. Then:
\begin{enumerate}\def\labelenumi{\rm (\arabic{enumi})}
\item
There exist a scalar measure $\nu$ on $\cA$ and a function $\Psi:\bR\to \B (\bC^m)$ (uniquely defined by $\nu$ up to $\nu$-a.e.) such that $\Psi (s)\geq 0$ $\nu$-a.e. on $\bR$, $\nu([\a,\b))<\infty$ and $\s(\b)-\s(\a)=\int\limits_{[\a,\b)}\Psi(s)\, d \nu $ for any finite interval $[\a,\b)\subset\bR$.
\item
The set  $\cL^2(\s;\bC^m)$ of all Borel-measurable functions $f=f(\cd):\bR\to \bC^m$ satisfying
\begin {equation*}
||f||_{\cL^2(\s;\bC^m)}^2=\int_\bR (d\s(s)f(s),f(s)):=\int_\bR(\Psi(s)f(s),f(s))_{\bC^m}\, d\nu <\infty
\end{equation*}
is a semi-Hilbert space with the semi-scalar product
\begin {equation*}
(f,g)_{\cL^2(\s;\bC^m)}=\int_\bR (d\s(s)f(s),g(s)):=\int_\bR(\Psi(s)f(s),g(s))_{\bC^m}\,d\nu, \quad f,g\in \cL^2(\s;\bC^m).
\end{equation*}
\end{enumerate}
\end{theorem}
\begin{definition}\label{def2.11}$\,$\cite{DunSch}
The Hilbert space $L^2(\s;\bC^m)$ is a Hilbert space of all equivalence classes in $\cL^2(\s;\bC^m)$ with respect to the seminorm $||\cd||_{\cL^2(\s;\bC^m)}$.
\end{definition}
In the following we denote by $\pi_\s$ the quotient map from $\cL^2(\s;\bC^m)$ onto $L^2(\s;\bC^m)$. Two functions $f_1,f_2\in\cL^2(\s;\bC^m)$ are said to be $\s$-equivalent if $\pi_\s f_1 = \pi_\s f_2$, i.e., if $\Psi(s)f_1(s)=\Psi(s)f_2(s)$ $\nu$-a.e on $\bR$.

With a distribution function $\s(\cd):\bR\to \B(\bC^m)$ one associates the $\B(\bC^m)$-valued measure  $\mu_\s$ on $\cA$ given by
\begin{gather}\label {2.14}
\mu_\s(B)=\int_B \Psi(s)\,d\nu, \quad B\in \cA.
\end{gather}
This measure is a  continuation of the measure $\mu_{0\s}$  on finite intervals $[\a,\b)\subset\bR$ defined by $\mu_{0\s}([\a,\b))=\s(\b)-\s(\a)$.

Let $\s(s)(\in \B (\bC^m))$ be a distribution function. For Borel measurable functions $Y(s)(\in \B (\bC^m,\bC^k))$ and $ g(s)(\in \bC^m)$ on $\bR$ we let
\begin{gather}\label {2.16}
\int_\bR Y(s)d\s(s)g(s):=\int_\bR Y(s)\Psi(s)g(s)\, d\nu \, (\in \bC^k)
\end{gather}
where $\nu$ and $\Psi(\cd)$ are defined in Theorem \ref{th2.10}, (1).
\section{Pseudospectral and spectral functions of Hamiltonian systems}
\subsection{Notations}
Let $\cI=[ a,b\rangle\; (-\infty < a< b\leq\infty)$ be an interval of the real line (the endpoint $b<\infty$  might be either
included  to $\cI$ or not). Denote by $AC(\cI;\bC^n)$ the set of functions $f(\cd):\cI\to \bC^n$ which are absolutely
continuous on each segment $[a,\b]\subset \cI$.

An operator-function $Y(\cd):\cI\to \B(\bC^n)$ is called locally integrable if $\int\limits_{[a,b']}||Y(t)||\, dt<\infty$ for each $b'\in\cI$. Assume that $\D(\cd):\cI\to \B(\bC^n)$ is a locally integrable function such that $\D(t)\geq 0$ a.e. on $\cI$. Denote  by $\lI$  the semi-Hilbert  space of  Borel measurable
functions $f(\cd): \cI\to \bC^n$ satisfying $||f(\cd)||_\D^2:=\int\limits_{\cI}(\D
(t)f(t),f(t))\,dt<\infty$ (see e.g. \cite[Chapter 13.5]{DunSch}).  The
semi-definite inner product $(\cd,\cd)_\D$ in $\lI$ is defined by $(f(\cd),g(\cd))_\D=\int\limits_{\cI}(\D (t)f(t),g(t))\,dt,\quad f(\cd),g(\cd)\in \lI$. Moreover, let $\LI$ be the Hilbert space of the equivalence classes in $\lI$ with respect
to the semi-norm $||\cd||_\D$. Denote also by $\pi_\D$  the quotient map from $\lI$ onto
$\LI$ and let $\wt \pi_\D\{f(\cd),g(\cd)\}:=\{\pi_\D f(\cd),
\pi_\D g(\cd)\}, \;\; \{f(\cd),g(\cd)\} \in (\lI)^2$. Clearly, $\ker \pi_\D$ coincides with the set of all Borel measurable functions $f(\cd):\cI\to \bC^n$ such that $\D(t) f(t)=0$ (a.e. on $\cI$).
\subsection{Hamiltonian  systems}
Let as above $\cI=[ a,b\rangle\; (-\infty < a< b\leq\infty)$ be an interval in $\bR$, let $p\in\bN$ and let $n=2p$.  Recall that a Hamiltonian  system of the dimension $n$ on an interval $\cI$ (with the regular endpoint $a$) is a system of differential equations
\begin {equation}\label{3.1}
J y'-B(t)y=\l\D(t)y, \quad t\in\cI, \quad \l\in\bC
\end{equation}
where $B(\cd)$ and $\D(\cd)$ are locally integrable
$\B (\bC^n)$-valued  functions on $\cI$ satisfying $B(t)=B^*(t)$ and  $\D(t)\geq 0$ for any  $t\in\cI$ and $J\in \B (\bC^n)$ is the operator given by \eqref{1.3}. Together with system  \eqref{3.1} we consider the inhomogeneous  system
\begin {equation}\label{3.3}
J y'-B(t)y=\D(t) f(t), \quad t\in\cI,
\end{equation}
where $f(\cd)\in \lI$. A function $y(\cd)\in\AC$ is a solution of \eqref{3.1} (\eqref{3.3}) if it satisfies  \eqref{3.1} (resp. \eqref{3.3})   a.e. on $\cI$. A function $Y(\cd,\l):\cI\to \B (\bC^k,\bC^n)$ is an
operator solution of  \eqref{3.1} if $y(t)=Y(t,\l)h$ is a
(vector) solution of \eqref{3.1}  for every $h\in \bC^k$.  In the sequel we denote by $Y_0(\cd)$ the $\B (\bC^n)$-valued  operator solution of the system
\begin {equation}\label{3.5.1}
J y'-B(t)y=0
\end{equation}
such that $Y_0(0)=I_n$. As is known, $Y_0(t)$ satisfies the identities
\begin{gather}\label{3.5.1.1}
Y_0^*(t)J Y_0(t)= J, \qquad  Y_0(t)J Y_0^*(t)= J
\end{gather}
By using the second identity in \eqref{3.5.1.1} one can easily verify that each solution $y(\cd)$ of \eqref{3.3} admits the representation
\begin{gather}\label{3.5.2}
y(t)=z(t) -Y_0(t)J \int_{[a,t]} Y_0^*(u)\D(u) f(u)\,du,
\end{gather}
where $z(\cd)\in\AC$ is the solution of \eqref{3.5.1} with $z(a)=y(a)$.

As it is known (see e.g.\cite{Kac03,LesMal03}) system
\eqref{3.1} gives rise to the maximal linear relations
$\tma$ and $\Tma$  in  $\lI$ and $\LI$ respectively. Namely, $\tma$ is the set of all pairs $\{y(\cd),f(\cd)\}\in(\lI)^2$ such that  $y(\cd)\in\AC$ and \eqref{3.3} holds a.e. on $\cI$, while $\Tma=\wt\pi_\D\tma$. Moreover for any $y(\cd),z(\cd)\in\dom\tma$ there exists the limit
\begin {equation*}
[y,z]_b:=\lim_{t \uparrow b}(J y(t),z(t)).
\end{equation*}
Next, define the linear relation $\cT_a$ in $\lI$ and the minimal  linear relation $\Tmi$ in $\LI$ by setting
\begin {equation*}
\cT_a=\{\{y(\cd),f(\cd)\}\in\tma: y(a)=0 \;\;\text{and}\;\;\, [y,z]_b=0 \;\;\text{for
every}\;\; z\in\dom\tma\}
\end{equation*}
and $\Tmi=\wt\pi_\D\cT_a$. Then $\Tmi$ is a closed symmetric linear
relation in $\LI$ and $\Tmi^*=\Tma$ \cite{Kac03, LesMal03,Mog12}.

The null manifold $\cN$ of the system \eqref{3.1} is  defined as a linear space of all solutions $y(\cd)$ of \eqref{3.5.1} such that $\D(t)y(t)=0$ (a.e. on $\cI$).

In the sequel we denote by $\cN_\l,\; \l\in\bC,$ the linear space of solutions of the  system \eqref{3.1} belonging to $\lI$. The  numbers $N_+=\dim \cN_i$ and $N_-=\dim\cN_{-i} $ are called the formal
deficiency indices of the system \eqref{3.1}. It was shown in  \cite{KogRof75,LesMal03} that  $N_\pm=\dim \cN_\l, \; \l\in\bC_\pm$
(i.e., $\dim \cN_\l$ does not depend on $\l$ in either $\bC_+$ or $\bC_-$) and $p\leq N_\pm \leq n$. Moreover, deficiency indices of $\Tmi$ are $n_\pm(\Tmi)=N_\pm-\dim \cN$.

Recall  that system \eqref{3.1} is called definite if $\cN=\{0\}$.
\begin{definition}\label{def3.1.0}
Let  $U\in\B (\bC^n,\bC^p)$ be an operator such that
\begin{gather}\label{3.20}
UJU^*=0 \;\; {\rm and} \;\; \ran U=\bC^p.
\end{gather}
System \eqref{3.1} is called $U$-definite if for each $y\in \cN$ the equality $U y(a)=0$ yields $y=0$.
System \eqref{3.1} is called $U$-definite  on an interval $\cI'\subset\cI$ if its restriction on $\cI'$ is $U$-definite.
\end{definition}
Clearly each  definite system is $U$-definite  for any $U$.

It was proved in \cite{KogRof75} that for each definite system there is a compact interval $[a,\b]\subset \cI$ such that the system is definite on $[a,\b]$. In the same way one  proves the following proposition.
\begin{proposition}\label{pr3.1.0.1}
If system \eqref{3.1} is $U$-definite, then there is a compact interval $[a,c]\subset \cI$ such that the system is U-definite on $[a,c]$.
\end{proposition}
\subsection{Pseudospectral and spectral functions}
Below we suppose that  $U\in\B (\bC^n,\bC^p)$ is an operator satisfying \eqref{3.20}. Then  the following assertion holds (see \cite[Lemma 3.3]{Mog17}).
\begin{assertion}\label{ass3.2}
The equality
\begin{gather}\label{3.22.1}
T=\{\wt\pi_\D \{y,  f\}: \{y,f\}\in\tma,\; Uy(a)=0 \;\;{\rm and}\; \;[y,z]_b=0,\; z\in\dom\tma \}
\end{gather}
defines a  (closed) symmetric extension $T$ of $\Tmi$. Moreover, $T^*=\wt\pi_\D \cT_*$, where $\cT_*$ is the linear relation in $\lI$ given by
\begin{gather}\label{3.22.2}
\cT_*=\{\{y(\cd),f(\cd)\}\in\tma:Uy(a)=0\}
\end{gather}
\end{assertion}

Clearly the domain of $\cT_*$ is
\begin{multline}\label{3.22.3}
\dom\cT_*=\{y(\cd)\in \AC\cap \lI:\, Jy'(t)-B(t)y(t)=\D(t)f_y(t)\\ (\text{a.e. on}\;\; \cI)\;\;
\text{with some} \;\; f_y(\cd)\in \lI
\;\; {\rm and}\;\;Uy(a)=0\} .
\end{multline}
Note that $f_y(\cd)$ in \eqref{3.22.3} is defined by $y(\cd)$ uniquely up to the equivalence with respect to the seminorm $||\cd||_\D$.

In what follows we put $\gH:=\LI$ and $\gH_0:=\gH\ominus \mul T$.  Since $T$ is a symmetric relation in $\gH$,  the decompositions
\begin{gather}\label{3.23}
\gH=\gH_0\oplus\mul T, \qquad T={\rm gr}T_0\oplus \wh\mul T
\end{gather}
hold with $ \wh\mul T =\{0\}\oplus \mul T$ and a (not necessarily densely defined) symmetric operator $T_0$ in $\gH_0$ (this operator is called the operator part of $T$).

Below we denote by $\cL',\;\cL_0$ and $\cD$ the linear manifolds in $\lI$ defined by
\begin{gather}
\cL'=\{f(\cd)\in\lI: \text{ there exists  a solution} \; y(\cd)\; \text{ of \eqref{3.3} such that}\qquad\label{3.23.1}\\
\qquad\qquad\qquad \D(t) y(t)=0\;\; (\text {a.e. on} \;\; \cI), \; U y(a)=0 \;\;{ \rm and} \;\; [y,z]_b=0,\;z\in\dom\tma\}\nonumber\\
\cL_0=\{f(\cd)\in\lI: (f(\cd),g(\cd))_\D=0 \;\; \text{for any}\;\; g(\cd)\in\cL'\}\label{3.24}\\
\cD=\{y(\cd)\in \dom \cT_*:f_y(\cd) \in \cL_0\}\label{3.24.1}
\end{gather}
Clearly, $\mul T=\pi_\D \cL'$ and $\gH_0=\pi_\D\cL_0$.

Let $\f_U(\cd,\l)(\in \B(\bC^p,\bC^n)),\; \l\in\bC,$ be the operator solution of \eqref{3.1} with the initial value $\f_U(a,\l)=-JU^*$. One can easily prove that for each function $f(\cd)\in\lI$ and each point $c\in\cI$ the equality
\begin{gather}\label{3.25}
\wh f_c(s)=\int_\cI \f_U^*(t,s)\D(t) \chi_{[a,c]}(t) f(t)\,dt
\end{gather}
defines a continuous function $f_c(\cd):\bR\to \bC^p$ (the integral in \eqref{3.25} is understood as the Lebesgue integral).
\begin{definition}\label{def3.3} $\,$ \cite{Mog15}
A distribution function $\s(\cd):\bR\to \B(\bC^p)$ is called a pseudospectral function of the system \eqref{3.1} if:

(i) for each function $f(\cd)\in\lI$ and each $c\in\cI$ one has $\wh f_c(\cd)\in\lS$ and  there exists a function $\wh f(\cd) \in \lS$ such that
\begin{gather}\label{3.26}
\lim_{c\uparrow b}||\wh f(\cd)- \wh f_c(\cd)||_{\lS}=0;
\end{gather}

(ii) $||\wh f(\cd)||_{\lS}=0$ for $f(\cd)\in\cL'$ and  the Parseval equality $||\wh f(\cd)||_{\lS}=||f(\cd)||_{\lI}$ holds for all $f(\cd)\in\cL_0$.
\end{definition}
Clearly, the function $\wh f(\cd)$ in Definition \ref{def3.3} is defined by $f(\cd)$ uniquely up to the $\s$-equivalence. This function is called the (generalized) Fourier transform of a function $f(\cd)\in\lI$.

Definition \eqref{3.26} of $\wh f(\cd)$ can be written as
\begin{gather}\label{3.26.1}
\wh f(s)=\int_\cI \f_U^*(t,s)\D(t)  f(t)\,dt,
\end{gather}
where the integral converges in the seminorm of $\lS$.
\begin{definition}\label{def3.3.1}
A distribution function $\s(\cd):\bR\to \B(\bC^p)$ is called a spectral function of the system \eqref{3.1} if for each function $f(\cd)\in\lI$ with compact support the corresponding Fourier transform \eqref{3.26.1} (with the Lebesgue integral in the right hand side) satisfies the Parseval equality $||\wh f(\cd)||_{\lS}=||f(\cd)||_{\lI}$.
\end{definition}
Clearly, for a spectral  function $\s(\cd)$ the   Fourier transform  \eqref{3.26.1} (with the integral convergent in $\lS$) satisfies the Parseval equality $||\wh f(\cd)||_{\lS}=||f(\cd)||_{\lI}$ for every $f(\cd)\in\lI$.
\begin{remark}\label{rem3.3.2}
If $\s(\cd)$ is a pseudospectral function, then the equality
\begin{gather}\label{3.26.2}
V_\s\wt f=\pi_\s \wh f(\cd), \quad \wt f\in\gH,
\end{gather}
where $\wh f(\cd)$ is the Fourier transform of a function  $f(\cd)\in\wt f$, defines a partial isometry $V_\s\in\B (\gH, \LS)$ such that $\ker V_\s=\mul T$ and $||V_\s\wt f||=||\wt f||, \; \wt f\in\gH_0$ (see \eqref{3.23}). Clearly, $V_\s$ is an isometry if and only if $\s(\cd)$ is a spectral function.
\end{remark}
\begin{definition}\label{def3.3.3}
A pseudospectral (spectral) function $\s(\cd)$ of \eqref{3.1} is called orthogonal if $\ran V_\s=L^2(\s;\bC^p)$.
\end{definition}
\begin{proposition}\label{pr3.4}$\,$ \cite{Mog15}
Let $\s(\cd)$ be a pseudospectral function of the system \eqref{3.1}. Then for each function $g(\cd)\in\lS$  the following holds:

{\rm (i)} for  each bounded Borel set $B\subset \bR$ the equality
\begin{gather}\label{3.26.4}
\ov g_B(t)=\int_\bR \f_U(t,s)\, d\s (s) \chi_B(s)g(s), \quad t\in\cI
\end{gather}
defines a function $\ov g_B(\cd)\in\lI$ (the integral in \eqref{3.26.4} exists as the Lebesgue integral, see \eqref{2.16})

{\rm (ii)} there exists a function $\ov g(\cd)\in\lI$ such that for each sequence $\{B_n\}_1^\infty$ of bounded Borel sets $B_n\subset\bR$ satisfying $B_n\subset B_{n+1}$ and $\mu_\s\left(\bR\setminus \bigcup_{n\in\bN} B_n \right )=0$ the following equality holds
\begin{gather}\label{3.27}
\lim_{n\to\infty}||\ov g(\cd)-\ov g_{B_n}(\cd)||_{\lI}=0.
\end{gather}
Equality \eqref{3.27} is written as $\ov g(t)=\int_\bR \f_U(t,s)\, d\s (s) g(s)$, where the integral converges in the seminorm of $\lI$.

Moreover, for each $\wt g\in\LS$ one has
\begin{gather}\label{3.28}
V_\s^* \wt g=\pi_\D \ov g(\cd)=\pi_\D \left (\int_\bR \f_U(\cd,s)\, d\s (s) g(s)\right), \quad g(\cd) \in \wt g.
\end{gather}
\end{proposition}
\begin{corollary}\label{cor3.5}
Let $\s(\cd)$ be a pseudospectral function of the system \eqref{3.1}, let $f(\cd)\in\cL_0$ and let $\wh f(\cd)$ be the Fourier transform of $f(\cd)$. Then
\begin{gather}\label{3.29}
f(t)=\int_\bR \f_U(t,s)\, d\s (s) \wh f(s),
\end{gather}
where the integral  converges in the seminorm of $\lI$.
\end{corollary}
\begin{remark}\label{rem3.6}
The equality \eqref{3.29} is called the inverse Fourier transform of a function $f(\cd)$.  Clearly, \eqref{3.29} is valid for each $f(\cd)\in\lI$ if and only if $\s(\cd)$ is a spectral function.
\end{remark}
\begin{remark}\label{rem3.7}
According to \cite{Mog15} a distribution function $\s(\cd):\bR\to \B (\bC^p)$ is called a $q$-pseudospectral function of the system \eqref{3.1} if the condition (i)  of Definition \ref{def3.3} is satisfied and the Fourier transform $V_\s$ of the form \eqref{3.26.2} is a  partial isometry from $\gH$ to $\LS$. According to \cite[Proposition 3.8]{Mog15} for each $q$-pseudospectral function $\s(\cd)$ one has $\mul T\subset \ker V_\s$. This implies that for a pseudospectral function $\s(\cd)$ the Fourier transform $V_\s$ has the minimally possible kernel $\ker V_\s$ among all $q$-pseudospectral functions and hence the inverse Fourier transform \eqref{3.29} is valid for functions $f(\cd)$ from the maximally possible set (namely, from the set $\cL_0$).  This facts justify our interest to pseudospectral functions.
\end{remark}

\begin{proposition}\label{pr3.9}$\,$ \cite{Mog15}
Assume that:

{\rm (A1)} system \eqref{3.1}has equal formal deficiency indices $N_+=N_-=:d$;

{\rm (A2)} $U\in\B(\bC^n,\bC^p)$ is an operator  satisfying  \eqref{3.20} and system \eqref{3.1} is $U$-definite;

{\rm (A3)} $\G_b=(\G_{0b}, \G_{1b})^\top:\dom \tma \to \bC^{d-p}\oplus \bC^{d-p}$ is a surjective  operator satisfying
\begin{gather*}
[y,z]_b=(\G_{0b}y,\G_{1b}z)- (\G_{1b}y,\G_{0b}z),\quad y,z \in \dom\tma
\end{gather*}
(such an operator exists in view of \cite[Lemma 4.1]{Mog15}).

Moreover, let $T$ be the symmetric extension \eqref{3.22.1} of $\Tmi$. Then:

{\rm (i)} for each pair $\{\wt y, \wt f\}\in T^*$  there exists a unique  function $y(\cd)\in\dom\cT_* $ such that $\pi_\D y(\cd)=\wt y$ and $\pi_\D f_y(\cd)=\wt f$;

{\rm (ii)} the  collection $\Pi=\{\bC^{d-p},\G_0,\G_1\}$ with operators $\G_j:T^*\to \bC^{d-p}$ given by
\begin{gather}\label{3.32}
\G_0\{\wt y,\wt f\}=\G_{0b} y, \qquad \G_1\{\wt y,\wt f\}=-\G_{1b} y, \quad \{\wt y,\wt f\}\in T^*
\end{gather}
is a boundary triplet for $T^*$ (in \eqref{3.32} $y(\cd)\in \wt y$ is a function from statement {\rm (i)}).
\end{proposition}
\begin{remark}\label{rem3.10}
In the case of the system \eqref{3.1} on a compact interval $\cI=[a,b]$ one has $d=2p$. In this case one can put $\G_b y=y(b),\; y\in \dom\tma$.
\end{remark}
\begin{theorem}\label{th3.12}$\,$\cite{Mog15}
Let the assumptions {\rm (A1)} and  {\rm (A2)} in Proposition \ref{pr3.9} be satisfied. Then the set of pseudospectral functions of the system \eqref{3.1} is not empty and  there exists a Nevanlinna operator function
\begin{gather}\label{3.33}
M(\l)=\begin{pmatrix} m_0(\l) & M_2(\l)\cr M_3(\l) & M_4(\l) \end{pmatrix}:\bC^p\oplus\bC^{d-p}\to \bC^p\oplus\bC^{d-p}, \quad \l\in\CR
\end{gather}
such that $M_4(\cd)\in R_u[\bC^{d-p}]$ and  the equalities
\begin {gather}
m_\tau(\l)=m_0(\l)-M_2(\l)(\tau(\l)+M_4(\l))^{-1}M_3(\l),
\quad\l\in\CR\label{3.36}\\
\s_\tau(s)=\lim\limits_{\d\to+0}\lim\limits_{\varepsilon\to +0} \frac 1 \pi
\int_{-\d}^{s-\d}\im \,m_\tau(u+i\varepsilon)\, du \label{3.37}
\end{gather}
establish a bijective correspondence $\s(\cd)=\s_\tau(\cd)$ between all functions $\tau=\tau(\cd)\in \wt R (\bC^{d-p})$ satisfying the admissibility condition
\begin{gather}\label{3.37.1}
\lim_{y\to \infty}\tfrac 1 {iy} (\tau(i y)+M_4(i y))^{-1}=\lim_{y\to \infty}\tfrac 1 {i y} (\tau^{-1}(i y)+M_4^{-1}(i y))^{-1} =0.
\end{gather}
and  all pseudospectral functions $\s(\cd)$. Moreover, the following statements hold: {\rm (i)} all functions $\tau(\cd)\in \wt R (\bC^{d-p})$ satisfy \eqref{3.37.1} if and only if $\mul T=\mul T^*$; {\rm (ii)} a pseudospectral function $\s_\tau(\cd)$ is orthogonal if and only if $\tau(\l)\equiv \t(=\t^*),\; \l\in\CR$.
\end{theorem}
Note that  the matrix $M(\l)$ in \eqref{3.33} is defined in terms of the boundary values of certain operator solutions of \eqref{3.1} at the endpoints $a$ and $b$ (see \cite[Proposition 4.9]{Mog15}).
\begin{definition}\label{def3.12.1}
A  function $\tau\in \wt R (\bC^{d-p})$ satisfying  \eqref{3.37.1} will be called an admissible boundary parameter.
\end{definition}
Clearly,  Theorem \ref{th3.12} gives a parametrisation of all pseudospectral  functions of the system \eqref{3.1} in terms  of the admissible boundary parameter $\tau$.
\begin{remark}\label{rem3.12.2}
The operator function $m_\tau (\cd)$ in \eqref{3.36} coincides with the $m$-function of the system \eqref{3.1} corresponding to the admissible  boundary parameter $\tau$ (see \cite{Mog15}). Note that $m_\tau(\cd)\in R [\bC^d]$  and \eqref{3.37} is the Perron-Stieltjes formula for $m_\tau$. In the case of the constant-valued admissible boundary parameter $\tau(\l)\equiv \tau(=\tau^*), \; \l\in\CR,$  the function $m_\tau(\cd)$ turns into the $m$-function (Titchmarsh - Weyl function) of the system in the sense of \cite{HinSha81,HinSch93}.
\end{remark}
\begin{proposition}\label{pr3.13}
{\rm (i)} For system \eqref{3.1} the following   equivalences are valid:
\begin{gather}\label{3.38}
\cL'\subset \ker \pi_\D \iff \cL_0=\lI \iff \mul T=\{0\}\iff \cD=\dom \cT_*
\end{gather}
If $\D(t)$ is invertible a.e on $\cI$, then all the relations in \eqref{3.38} hold.

{\rm (ii)} Let for system \eqref{3.1} the assumptions {\rm (A1)} and {\rm (A2)} be satisfied.  Then the set of spectral functions of the system \eqref{3.1} is not empty if and only if at least one (and hence all) of the equivalent conditions in \eqref{3.38} are satisfied. Moreover, in this case the sets of  spectral and pseudospectral functions coincide and hence Theorem \ref{th3.12} holds for spectral functions.
\end{proposition}
\begin{proof}
(i) The first and second equivalences in \eqref{3.38} are obvious. Next, by \eqref{3.23} one has $T^*=T_0^*\oplus \wh\mul T$, where $T_0^*\in\C (\gH_0)$. This yields the third  equivalence in \eqref{3.38}.

Statement (ii) directly follows from \cite[Theorem 5.12]{Mog15}.
\end{proof}
\section{Uniform convergence of the inverse Fourier transform for Hamiltonian systems}
\begin{lemma}\label{lem4.14}
Suppose that  system \eqref{3.1} is given on a compact interval $I=[a,b]$ and satisfies the assumption {\rm (A2)}  in Proposition \ref{pr3.9}. Let  $\cN_0'$ be a  linear space of all solutions $y(\cd)$ of the system \eqref{3.5.1} satisfying $U y(a)=0$ (clearly, $\cN_0'\subset \lI$ ), let $y(\cd)\in \cN_0'$ and let $\{y_n(\cd)\}_1^\infty$ be a sequence of functions $y_n(\cd)\in \cN_0'$ such that $||y_n(\cd)-y(\cd)||_\D \to 0$. Then
\begin{gather}\label{3.40}
\lim_{n\to \infty} \sup_{t\in\cI} ||y(t)-y_n(t)||=0.
\end{gather}
\end{lemma}
\begin{proof}
Let $y(\cd)\in \cN_0'$ and $(y(\cd),y(\cd))_\D=0$. Then $\D(t)y(t)=0$ and hence $y(\cd)\in\cN$. Since $U y(a)=0$ and the system is $U$-definite, the equality $y=0$ holds. Thus $\cN_0'$ is a finite dimensional Hilbert space with the inner product $(\cd,\cd)_\D$. Clearly, the relation $\cN_0'\ni y(\cd) \to y(0)\in \ker U$ defines a linear isomorphism of $\cN_0'$ onto $\ker U$. Therefore the condition  $||y_n(\cd)-y(\cd)||_\D \to 0$ yields $y_n(0)\to y(0)$, which implies \eqref{3.40}.
\end{proof}
\begin{proposition}\label{pr3.14}
Let system \eqref{3.1} satisfies the assumption {\rm (A2)}. Assume  also that $\{y(\cd),f(\cd)\}\in \cT_*$ and let $\{y_n(\cd)\}_1^\infty$ and $\{f_n (\cd)\}_1^\infty$ be sequences of functions $y_n(\cd), f_n(\cd)\in\lI$ such that $\{y_n(\cd), f_n(\cd)\}\in\cT_*$ and $||y_n(\cd)-y(\cd)||_\D\to 0, \; ||f_n(\cd)-f(\cd)||_\D\to 0$. Then for each compact interval $[a,c]\subset \cI$ one has
\begin{gather}\label{3.41}
\lim_{n\to \infty} \sup_{t\in [a, c]} ||y(t)-y_n(t)||=0.
\end{gather}
\end{proposition}
\begin{proof}
(i) First  suppose that system \eqref{3.1} is given  on a compact interval $\cI=[a,b]$. Since $y(\cd)$ and $y_n(\cd)$ are solutions of \eqref{3.3} with $f(\cd)$ and $f_n(\cd)$ respectively, it follows from \eqref{3.5.2} that
\begin{gather}\label{3.42}
y(t)=z(t)+g(t), \qquad y_n(t)=z_n(t)+g_n(t), \quad t\in\cI,
\end{gather}
where $z(\cd)$ and $z_n(\cd)$ are solutions of \eqref{3.5.1} with $z(a)= y(a)$ and $z_n(a)=y_n(a)$ and
\begin{gather}\label{3.43}
g(t)=-Y_0(t)J \int_a^t Y_0^*(u)\D(u) f(u)\,du, \quad g_n(t)=-Y_0(t)J \int_a^t Y_0^*(u)\D(u) f_n(u)\,du.
\end{gather}
Let
\begin{gather*}
r(t)=\int_a^t Y_0^*(u)\D(u) f(u)\,du, \qquad r_n(t)=\int_a^t Y_0^*(u)\D(u) f_n(u)\,du.
\end{gather*}
Then for any $t\in\cI$ and $h\in\bC^n$ one has
\begin{gather*}
|(r(t)-r_n(t),h)|=\left | \int_a^t \left ( Y_0^*(u)\D(u)( f(u) - f_n(u)), h \right) \, du \right |=\\
\left | \int_a^t \left (\D(u)( f(u) - f_n(u)),Y_0(u) h \right) \, du \right |=\left | \left ( f(\cd)-f_n(\cd), Y_0(\cd)h \right )_{\cL_\D^2([a,t])}\right |\leq\\
||f(\cd)-f_n(\cd)||_{\cL_\D^2([a,t])}\cd ||Y_0(\cd)h||_{\cL_\D^2 ([a,t])}\leq ||f(\cd)-f_n(\cd)||_\D\cd ||Y_0(\cd)h||_\D.
 \end{gather*}
This implies that
\begin{gather*}
\lim_{n\to\infty}\sup_{t\in\cI}|(r (t)-r_n(t),h) |=0,\quad h\in\bC^n
\end{gather*}
and hence
\begin{gather}\label{3.44}
\lim_{n\to\infty}\sup_{t\in\cI} ||r (t)-r_n(t)||=0.
\end{gather}
Since by \eqref{3.43} $g(t)-g_n(t)=-Y_0(t)J (r(t)-r_n(t)$ and the operator function $Y_0(t)$ is bounded in $\cI$, it follows from \eqref{3.44} that
\begin{gather}\label{3.45}
\lim_{n\to\infty}\sup_{t\in\cI} ||g (t)- g_n(t)||=0.
\end{gather}
Therefore $||g(\cd)-g_n(\cd)||_\D \to 0$ and by \eqref{3.42} $||z(\cd)-z_n(\cd)||_\D\to 0$. Since $U z(a)=U y(a)$ and $U z_n(a)=U y_n(a)$, it follows from \eqref{3.22.2} that $Uz(a)=U z_n(a)=0$. Therefore $z(\cd), z_n (\cd)\in \cN_0'$ (for $\cN_0'$ see Lemma \ref{lem4.14}) and by Lemma \ref{lem4.14}
\begin{gather}\label{3.46}
\lim_{n\to\infty}\sup_{t\in\cI} ||z (t)- z_n(t)||=0.
\end{gather}
Now combining \eqref{3.42} with \eqref{3.45} and \eqref{3.46} we arrive at the equality
\begin{gather*}
\lim_{n\to\infty}\sup_{t\in\cI} ||y (t)- y_n(t)||=0.
\end{gather*}

(ii) Now consider  system \eqref{3.1} on an  interval $\cI=[a,b), \; b\leq \infty$. According to Proposition \ref{pr3.1.0.1} there is a segment $I_0=[a,c_0]\subset \cI$ such that the system is $U$-definite on $\cI_0$. Let $\cI_1=[a,c]$ be a segment in $\cI$ and let $\cI'=[a,c']\subset \cI$ be a segment such that $\cI_0\subset \cI'$ and $\cI_1\subset \cI'$. Then the system is $U$-definite on $\cI'$. Let $\cT_*'$ be the linear relation \eqref{3.22.2} corresponding to the restriction of the system \eqref{3.1} on $\cI'$ and let $\ov y(\cd),\; \ov y_n(\cd),\; \ov f(\cd),\; \ov f_n(\cd)$ be restrictions of the functions  $ y(\cd),\;  y_n(\cd),\;  f(\cd),\;  f_n(\cd)$ on $\cI'$ respectively. Clearly, $\{\ov y(\cd), \ov f(\cd)\}\in \cT_*', \; \{\ov y_n(\cd), \ov f_n(\cd)\}\in \cT_*' $ and $||\ov y_n(\cd)-\ov y(\cd)||_{\cL_\D^2(\cI',\bC^n)}\to 0$,  $||\ov f_n(\cd)-\ov f(\cd)||_{\cL_\D^2(\cI',\bC^n)}\to 0$. Therefore by statement (i)
\begin{gather}\label{3.47}
\lim_{n\to\infty}\sup_{t\in\cI'} ||y (t)- y_n(t)||=0.
\end{gather}
and the inclusion $\cI_1 \subset \cI'$ implies that \eqref{3.47} holds with $\cI_1=[a,c]$ instead of $\cI'$.
\end{proof}
Let $T$ be a symmetric relation \eqref{3.22.1} and let $\Pi=\{\bC^{d-p}, \G_0,\G_1\}$ be a boundary triplet \eqref{3.32} for $T^*$. Moreover, let $\tau\in \wt R(\bC^{d-p})$ be an admissible boundary parameter and let $\wt T_\tau=\wt T_\tau^*$ be the corresponding exit space extension of $T$ (see Theorem \ref{th2.6}, (ii)). Assume that $\wt T_\tau$ is a linear relation in a Hilbert space $\wt \gH\supset \gH$. Then according to \cite[Proposition 5.3]{Mog15} $\mul \wt T_\tau=\mul T$ and the equalities \eqref{3.23} for $\wt T_\tau$ take the form
\begin{gather}\label{3.48}
\wt\gH=\wt\gH_0\oplus\mul T, \qquad \wt T_\tau={\rm gr} \wt T_{0\tau}\oplus \wh\mul T,
\end{gather}
where $\wt\gH_0=\wt\gH \ominus \mul T$ and $\wt T_{0\tau}$ is a self-adjoint operator in $\wt\gH_0$.

Combining \eqref{3.48} with \eqref{3.23} one obtains that $\gH_0\subset \wt \gH_0$ and $\wt T_{0\tau}$ is an exit space extension of $T_0$.
\begin{proposition}\label{pr3.15}
Let for system \eqref{3.1} the assumptions {\rm (A1)} and  {\rm (A2)} in Proposition \ref{pr3.9} be satisfied. Moreover, let $\tau \in \wt R(\bC^{d-p})$ be an admissible boundary parameter, let $\s(\cd)=\s_\tau(\cd)$ be a pseudospectral function of the system \eqref{3.1}, let $\wt T_\tau=\wt T_\tau^*$ be an exit space extension of $T$ in the Hilbert space $\wt \gH\supset \gH$, let $\wt T_{0\tau}$ be the operator part of $\wt T_\tau$ (see  decompositions \eqref{3.48}) and let $E(\cd)$ be the orthogonal spectral measure of $\wt T_{0\tau}$. Next assume that $\{B_n\}_1^\infty$ is a sequence of bounded Borel sets $B_n\subset \bR$ such that $B_n\subset B_{n+1}$ and let $B:=\bigcup_{n\in\bN} B_n$. Let $\{y(\cd),f(\cd)\}\in\cT_*$ be a pair of functions such that $\wt y:=\pi_\D y(\cd)\in \dom \wt T_{0\tau}\cap \gH_0$ and $\wt f:=\pi_\D f(\cd) =P_{\gH_0}\wt T_{0\tau}\wt y$ (for $\gH_0$ see \eqref{3.23}) and let $\wt y_B:=P_{\gH_0} E(B)\wt y, \; \wt f_B:=P_{\gH_0} E(B) \wt T_{0\tau}\wt y$. Then:

{\rm(i)} $\{\wt y_B, \wt f_B\}\in T^*$ and hence there exists a pair of functions $\{ y_B(\cd),  f_B(\cd)\}\in \cT_*$  such that $\pi_\D y_B(\cd)= \wt y_B$ and $\pi_\D f_B(\cd)= \wt f_B$;

{\rm(ii)} if $\wh y(\cd)$ is the Fourier transform of $y(\cd)$, then for each compact interval $[a,c]\subset \cI$ one has
\begin{gather}\label{3.49}
\lim_{n\to\infty}\sup_{t\in [a,c]}\left| \left|y_B(t)-\int_\bR \f_u(t,s)d\s (s)\chi_{B_n}(s)\wh y(s)\right |\right |=0.
\end{gather}

{\rm(iii)} if in addition $\mu_\s(\bR\setminus B)=0$, then for each $[a,c]\subset\cI$ one has
\begin{gather}\label{3.49.0}
\lim_{n\to\infty}\sup_{t\in [a,c]}\left| \left|y(t)-\int_\bR \f_u(t,s)d\s (s)\chi_{B_n}(s)\wh y(s)\right |\right |=0.
\end{gather}
\end{proposition}
\begin{proof}
(i) Since $\{E(B) \wt y, E(B) \wt T_{0\tau} \wt y\}\in {\rm gr} \wt T_{0\tau}$, it follows from \eqref{2.12.1} that $\{\wt y_B, \wt f_B\}\in \Phi (\wt T_{0\tau})$ and hence $\{\wt y_B, \wt f_B\}\in T_0^*$. Since obviously $T_0^*\subset T^*$, this implies that $\{\wt y_B, \wt f_B\}\in T^*$.

(ii) Let $\cK:=V_{\s} \gH_0(\subset \LS)$ and let $\wt V_0 \in \B (\gH_0,\cK)$ be a unitary operator given by $\wt V_0 \wt f=V_{\s} \wt f, \; \wt f\in\gH_0 $. Denote by $\L_\s$ the
multiplication operator  in $\LS$ defined by
\begin{gather*}
\dom \Lambda_\s=\{\wt g\in \LS:s g(s)\in \LS \;\;\text{for some (and hence for all)}\;\; g(\cd)\in \wt g\}\nonumber\\
\Lambda_\s \wt g=\pi_{\s}(sg(s)), \;\; \wt g\in\dom\Lambda_\s,\quad g(\cd)\in\wt g.
\end{gather*}
As is known, $\Lambda_\s^*=\Lambda_\s$ and the orthogonal spectral measure $E_\s(\cd)$ of $\Lambda_\s$ is
\begin {equation}\label{3.49.1}
E_\s(B)\wt g= \pi_\s (\chi_B(\cd)g(\cd)), \quad B\in\cA,\;\; \wt g \in  L^2(\s;\bC^n),\;\; g(\cd)\in \wt g.
\end{equation}

According to \cite[Proposition 5.6]{Mog15} there  exists a unitary operator $\wt V \in \B  (\wt\gH_0,\LS)$ such that $\wt V\up \gH_0=V_{\s}\up\gH_0 $ and the operators  $\wt T_{0\tau}$ and $\L_\s$ are unitarily equivalent by means of $\wt V$. This implies that
\begin{gather}
P_{\gH_0} E(B_n)\up \gH_0=\wt V_0^* P_{\cK} E_\s(B_n)\wt V_0\label{3.49.2}\\
P_{\gH_0} E(B_n) \wt T_{0\tau}\up (\dom \wt T_{0\tau}\cap\gH_0) =\wt V_0^* P_{\cK} E_\s(B_n)\L_\s \wt V_0 \up(\dom \wt T_{0\tau}\cap\gH_0)\label{3.49.3}
\end{gather}
Since  $\wt V_0^* P_\cK\wt g=V_\s^*\wt g, \; \wt g\in\LS$, the equalities  \eqref{3.49.2}  and \eqref{3.49.3} can be written as
\begin{gather}
P_{\gH_0} E(B_n)\up \gH_0=V_\s^*  E_\s(B_n)V_\s \up \gH_0\label{3.49.4}\\
 P_{\gH_0} E(B_n)\wt T_{0\tau}\up (\dom \wt T_{0\tau}\cap\gH_0)=V_\s^* E_\s(B_n)\L_\s V_\s \up (\dom \wt T_{0\tau}\cap\gH_0).\label{3.49.5}
\end{gather}
Let $\wt y_n:=P_{\gH_0} E(B_n)\wt y$ and $\wt f_n:=P_{\gH_0} E(B_n)\wt T_{0\tau}\wt y$. Then by \eqref{3.49.4} and \eqref{3.49.5} one has
\begin{gather}\label{3.50}
\wt y_n=V_\s^*  E_\s(B_n)V_\s \wt y, \qquad \wt f_n=V_\s^*  E_\s(B_n)\L_\s V_\s \wt y.
\end{gather}
Combining \eqref{3.50} with \eqref{3.26.2}, \eqref{3.49.1} and \eqref{3.28} one gets $\wt y_n=\pi_\D y_n(\cd)$ and $\wt f_n=\pi_\D f_n(\cd)$, where $y_n(\cd)$ and $f_n(\cd)$ are functions from $\lI$ given by
\begin{gather*}
y_n(t)=\int_\bR \f_U(t,s) d\s(s)\chi_{B_n}(s)\wh y(s), \qquad f_n(t)=\int_\bR s \f_U(t,s) d\s(s) \chi_{B_n}(s)\wh y(s).
\end{gather*}
It was shown in the proof of Proposition 5.5 in \cite{Mog15} that $\{y_n(\cd),f_n(\cd)\}\in\cT_*$. Moreover, since $||\wt y_n - \wt y_B||_\gH \to 0$ and $||\wt f_n - \wt f_B||_\gH \to 0$, it follows that $|| y_n(\cd) - y_B(\cd)||_\D \to 0$ and $|| f_n(\cd) -  f_B (\cd)||_\D \to 0$. Therefore by Proposition \ref{pr3.14} for each segment $[a,c]\subset\cI$ the equality \eqref{3.49} is valid.

(iii) Assume that $\mu_\s(\bR\setminus B)=0$. Since the operators $E_\s(B)$ and $E(B)$ are unitarily equivalent, this implies that $E(\bR\setminus B)=0$ and hence $E(B)=I_{\wt\gH_0}$. Therefore $\wt y_B=\wt y, \; \wt f_B=\wt f$ and, consequently, $\pi_\D y(\cd)=\pi_\D  y_B(\cd), \;\pi_\D f(\cd)=\pi_\D  f_B(\cd)$. Thus by Proposition \ref{pr3.9}, (i) $y(\cd)=y_B(\cd)$ and \eqref{3.49} yields \eqref{3.49.0}.
\end{proof}
The main results of this section are given in the following two theorems.
\begin{theorem}\label{th3.16}
Let for system \eqref{3.1} the assumptions {\rm (A1) -- (A3)} in Proposition \ref{pr3.9} be satisfied   and let $\cD\subset \dom\cT_*$ be the linear manifold \eqref{3.24.1}. Assume also that $\tau \in \wt R (\bC^{d-p})$ is an admissible boundary parameter, $\s(\cd)=\s_\tau(\cd)$ is a pseudospectral function of the system and $\eta_\tau\in \C (\bC^{d-p})$ is the linear relation defined in Theorem \ref{th2.9}. Then for each function $y(\cd)\in\cD$ satisfying the boundary condition $\{\G_{0b}y(\cd), -\G_{1b}y(\cd)\} \in \eta_\tau$ the following statements hold:

{\rm (i)} If $\wh y(\cd)$ is the Fourier transform of $y(\cd)$, then $\int\limits_\bR ||\f_U(t,s)\Psi (s) \wh y(s)||\, d\nu<\infty,\; t\in\cI,$ and the inverse transform for $y(\cd)$ is
\begin{gather}\label{3.52}
y(t)=\int_{\bR} \f_U(t,s) d\s(s) \wh y(s)=\int_{\bR} \f_U(t,s) \Psi (s)\wh y(s) d\nu, \quad t\in\cI.
\end{gather}
Here $\Psi(\cd)$ and $\nu $ are the operator function and Borel measure for $\s(\cd)$ defined in Theorem \ref{th2.10} (the integral in \eqref{3.52} exists as the Lebesgue integral).

{\rm (ii)}  The integral in \eqref{3.52} converges uniformly on each compact interval $[a,c]\subset \cI$, that is for each  sequence $\{B_n\}_1^\infty$ of bounded Borel sets $B_n\subset \bR$ satisfying  $B_n\subset B_{n+1}$ and $\mu_\s\left(\bR\setminus \bigcup_{n\in\bN} B_n  \right )=0$   the equality \eqref{3.49.0} holds. This implies that
\begin{gather}\label{3.53}
\lim_{{\a\to - \infty}\atop{\b\to +\infty}} \sup_{t\in [a,c]}\left |\left | y(t)-\int_\bR \f_U(t,s) d\s(s) \chi_{[\a,\b]}(s)  \wh y(s)\right|\right |=0.
\end{gather}
\end{theorem}
\begin{proof}
Assume that $y(\cd)\in \cD$ and $\{\G_{0b}y(\cd), -\G_{1b}y(\cd)\} \in \eta_\tau$. Then according to \eqref{3.22.3} $\{y(\cd),f_y(\cd)\}\in \cT_*$ with some $f_y(\cd)$ and hence the pair $\{\wt y, \wt f\}=\wt\pi_\D \{y(\cd),f_y(\cd)\}$  belongs to $T^*$. Let $\Pi=\{\bC^{d-p}, \G_0, \G_1 \}$ be the boundary triplet \eqref{3.32} for $T^*$. Then $\{\G_0\{\wt y,\wt f\}, \G_1\{\wt y,\wt f\}\}\in \eta_\tau$ and by Theorem  \ref{th2.9} $\{\wt y,\wt f\}\in C(\wt T_\tau)$, where $C(\wt T_\tau)$ is the compression of the exit space extension $\wt T_\tau =\wt T_\tau^* $ of $T$ with $\mul \wt T_\tau = \mul T$. One can easily verify that
\begin{gather}\label{3.54}
 C(\wt T_\tau)={\rm gr }\, C(\wt T_{0\tau})\oplus\mul T
\end{gather}
where $C(\wt T_{0\tau}) = P_{\gH_0}\wt T_{0\tau}\up \gH_0 \cap \dom T_{0\tau} $ is the compression of the operator part  $\wt T_{0\tau}$ of $\wt T_\tau$ (see \eqref{3.48}). Since $f_y(\cd)\in\cL_0$, it follows that $\wt f\in\gH_0$ and by \eqref{3.54} $\{\wt y, \wt f\}\in  {\rm gr }\, C(\wt T_{0\tau})$, that is $\wt y\in \dom \wt T_{0\tau}\cap \gH_0$ and $\wt f=P_{\gH_0} \wt T_{0\tau} \wt y$. Therefore  by Proposition \ref{pr3.15} for any $t\in\cI$ and for any sequence $\{B_n\}_1^\infty$ of bounded Borel sets $B_n\subset \bR$ satisfying $B_n\subset B_{n+1}$ there exists $C>0$ such that $\left|\left |\int\limits_{\bR} \f_u(t,s)\Psi(s)\chi_{B_n}(s)\wh y(s)  d\nu(s)\right |\right |\leq C$. Hence $\int\limits_{\bR} ||\f_u(t,s)\Psi(s) \wh y(s) || d\nu(s)< \infty$ and Proposition \ref{pr3.15}, (iii) yields \eqref{3.52} and statement (ii).
\end{proof}
\begin{theorem}\label{th3.16.1}
Let the assumptions be the same as in Theorem \ref{th3.16}. Moreover, let at least one (and hence all) of the equivalent conditions in  \eqref{3.38} be satisfied (in particular this assumption is fulfilled if $\D(t)$ is invertible a.e. on $\cI$). Then  $\s(\cd)=\s_\tau (\cd)$ is a spectral function and  statements {\rm (i)} and {\rm (ii)} of Theorem \ref{th3.16} hold for any function $y(\cd)\in \AC\cap\lI$  such that:

{\rm (a)} the equality $Jy'(t)-B(t)y(t)=\D(t) f_y(t)$ (a.e. on $\cI$) holds with some $f_y(\cd)\in \lI$;

{\rm (b)} the   boundary conditions
\begin{gather*}
U y(a)=0, \qquad \{\G_{0b}y(\cd), -\G_{1b}y(\cd)\} \in \eta_\tau
\end{gather*}
are satisfied.
\end{theorem}
\begin{proof}
It follows from Proposition \ref{pr3.13}, (ii) that $\s(\cd)$ is a spectral function. Next, assume that $y(\cd)$ satisfies the conditions of the theorem. Then $y(\cd)\in \dom\cT_*$ and the last condition in \eqref{3.38} yields $y(\cd)\in \cD$.  Moreover,
$\{\G_{0b}y(\cd),$ $ -\G_{1b}y(\cd)\}\in \eta_\tau$ and by Theorem \ref{th3.16} statements {\rm (i)} and {\rm (ii)} of this theorem hold.
\end{proof}
\section{Uniform convergence of the inverse Fourier transform for differential equations}
\subsection{Preliminary results}
In this section we apply the above results to  ordinary differential operators of an even order  on an interval  $\cI=[a,b\rangle \; (-\infty<a<b\leq \infty)$ with the regular endpoint $a$.

Assume that
\begin {equation}\label{4.1}
l[y]= \sum_{k=1}^r  (-1)^k  (p_{r-k}(t)y^{(k)})^{(k)} + p_r(t) y
\end{equation}
is a symmetric differential expression  of  an even order $n=2r$ with
 operator valued coefficients $p_j(\cd):\cI\to \B(\bC^m)$ satisfying $p_0^{-1}(t)\in \B(\bC^m)$ and $ p_j(t)=p_j^*(t), \; t\in\cI$. Moreover, it is assumed that the operator-functions $p_0^{-1}(t)$ and  $p_j(t),\;j\in\{1,\dots , r\})$ are locally integrable.

The quasi-derivatives $y^{[j]}(\cd), \; j\in \{0,\; \dots,\; 2r\},$ of a function $y(\cd):\cI\to \bC^m$ are defined as follows \cite{Wei, KogRof75}:
\begin{gather}
y^{[j]}=y^{(j)}, \quad j\in \{0,1, \dots, r-1\}, \qquad
y^{[r]}=p_0 y^{(r)}   \label{4.2}\\
y^{[r+j]}= - (y^{[r+j-1]})' + p_j y^{(r-j)},\;\;\;j\in \{1,\dots. r\}\label{4.4}
\end{gather}
The quasi-derivatives $Y^{[j]}(\cd)$ of an operator-valued function $Y(\cd):\cI\to \B (\bC^\nu, \bC^m)$ are defined by \eqref{4.2} -- \eqref{4.4} with $Y$ instead of $y$.

Denote by $\dom l$ the set of all functions $y(\cd):\cI\to \bC^m$ such that $y^{[j]}(\cd)\in\ACm $ for all $j\in\{0,1, \dots, 2r-1\}$ and let $l[y]=y^{[2r]},\; y\in\dom l$.

Next assume that $\D(\cd):\cI\to\B (\bC^m)$ is a locally integrable  operator function satisfying $\D(t)\geq 0$  for any  $t\in\cI$. We consider the differential equation
\begin{gather}\label{4.5}
l[y]=\l \D(t) y, \quad t\in\cI,\;\;\l\in\bC
\end{gather}
and the corresponding inhomogeneous equation
\begin{gather}\label{4.6}
l[y]= \D(t) f(t), \quad t\in\cI,
\end{gather}
where $f(\cd)\in \lI$.

A function $y(\cd)\in\dom l$ is a solution of \eqref{4.5} (\eqref{4.6}), if it satisfies \eqref{4.5} (resp. \eqref{4.6}) a.e. on $\cI$. An operator function  $Y(\cd):\cI\to\B(\bC^\nu,\bC^m)$ is a solution of \eqref{4.5}, if the quasi-derivatives $Y^{[j]}(\cd), \;j\in \{0,1,\dots, 2r-1\}, $ are absolutely continuous on each segment $[a,c]\subset \cI$ and the equality $Y^{[2r]}(t)=\l \D(t)Y(t)$ holds a.e. on $\cI$.
\begin{definition}\label{def4.0}
Differential equation \eqref{4.5} is called regular if it is given on a compact interval $\cI=[a,b]$ (this implies that $\int\limits_{\cI} | | (p_0(t))^{-1} ||\, dt<\infty$, $\int\limits_{\cI} |  | p_k(t) ||\, dt<\infty, \; k\in\{1, 2, \dots , r\},$ and   $\int\limits_{\cI} ||\D(t)| | \, dt <\infty$).
\end{definition}
With a function $y(\cd)\in\dom l$ one associates a function $\bold y(\cd):\cI\to (\bC^m)^{2r}$, given by
\begin{gather}\label{4.6.1}
\bold y(t)=y(t) \oplus y^{[1]}(t) \oplus \dots \oplus y^{[r-1]}(t)\oplus y^{[2r-1]}(t) \oplus  y^{[2r-2]}(t)\oplus \dots\oplus y^{[r]}(t).
\end{gather}
With an operator solution $Y(\cd):\cI\to \B (\bC^\nu,\bC^m)$ of \eqref{4.5} one associates the operator function $\bold Y(\cd):\cI\to \B \left(\bC^\nu,(\bC^m)^{2r}\right )$ given by
\begin{gather}\label{4.7}
\bold Y(t)=(Y(t),\, \dots, \, Y^{[r-1]}(t),\,Y^{[2r-1]}(t),\, \dots, \, Y^{[r]}(t))^\top(\in \B (\bC^\nu,(\bC^m)^{2r}).
\end{gather}

Equation \eqref{4.5} gives rise to the maximal linear relations $\sma$ in $\lI$ and $\Sma$ in $\LI$ defined as follows: $\sma$ is the set of all pairs $\{y(\cd), f(\cd)\}\in (\lI)^2$ such that $y(\cd)\in\dom l$ and \eqref{4.6} holds a.e. on $\cI$, while  $\Sma=\wt\pi_\D \sma$.

It turns out that the equation \eqref{4.5} is equivalent in fact to a
certain  Hamiltonian system. More precisely, the following
proposition is implied by the results of \cite{KogRof75}.
\begin{proposition}\label{pr4.1}
Let $l[y]$ be the expression \eqref{4.1} and let
\begin{gather}
\cJ=\begin{pmatrix} 0 &  -I_{mr}  \cr I_{mr}  & 0
\end{pmatrix}: (\bC^m)^r \oplus (\bC^m)^r \to (\bC^m)^r \oplus (\bC^m)^r \label{4.7.1}\\
\wt\D(t)=\begin{pmatrix} \D(t) & 0 & \cdots & 0 \cr 0  & 0 & \cdots &0 \cr \vdots& \vdots &\ddots & \vdots \cr 0 & 0 & \cdots &0\end{pmatrix}:\underbrace{\bC^m\oplus\bC^m\oplus\cdots \, \oplus \bC^m}_{2r\;\;{\rm times}}\to\underbrace{\bC^m\oplus\bC^m\oplus\dots\, \oplus \bC^m}_{2r\;\;{\rm times}}\label{4.7.2}
\end{gather}
where $\D(t)$ is taken from \eqref{4.5}.  Then there exists a
locally integrable   operator function $ B(t)= B^*(t)(\in \B ((\bC^m)^{2r}), \; t\in \cI,$ (defined in terms of $p_j$ and $q_j$) such that the Hamiltonian  system
\begin {equation}\label{4.8}
 \cJ \bold y'- B(t)\bold y=\l\wt\D(t) \bold y, \quad t\in\cI, \;\;\l\in\bC
\end{equation}
and the corresponding inhomogeneous system
\begin {equation}\label{4.9}
 J \bold y'-B(t) \bold y= \wt\D(t)\dot f (t), \quad t\in\cI
\end{equation}
possesses  the following properties:

{\rm (i)} The relation $Y(\cd,\l)\to \bold Y(\cd,\l)$, where $\bold Y(\cd,\l)$ is given by \eqref{4.7}, gives a bijective correspondence between all  $\B (\bC^\nu,\bC^m)$-valued operator solutions $Y(\cd, \l)$ of \eqref{4.5} and all $\B (\bC^\nu,(\bC^m)^{2r})$-valued operator solutions  $\bold Y(\cd,\l)$ of \eqref{4.8}.

{\rm (ii)} Let $\tma$ be the maximal linear relation in
$\cL_{\wt\D}^2(\cI; (\bC^m)^{2r})$ induced by system \eqref{4.8}. Then the equality $\cU_1\{y(\cd),f(\cd)\}=\{\bold y(\cd), \dot f(\cd)\}, \; \{y(\cd),f(\cd)\}\in \sma,$ where \begin{gather}\label{4.10}
\dot f(t)=f(t)\oplus \, \underbrace{0\oplus \,\dots \,\oplus 0}_{2r-1\;\;{\rm times}}(\in (\bC^m)^{2r}),
\end{gather}
defines  a
bijective linear operator $\cU_1$ from $\sma$ onto $\tma$.

{\rm (iii)} Let  $\Tma$ be the maximal relation
in $L_{\wt\D}^2(\cI;(\bC^m)^{2r})$ induced by system \eqref{4.8}. Then
the equality $\cU_2 \wt f=\pi_{\wt\D} \dot f(\cd),\; \wt f\in \LIW, \; f(\cd)\in\wt f,$ defines a unitary operator $\cU_2$ from $\LIW$ onto $L_{\wt\D}^2(\cI; (\bC^m)^{2r})$ such that
\begin{gather}
(\cU_2\oplus \cU_2)\Sma=\Tma.\label{4.12}
\end{gather}
\end{proposition}
Let $\tma,\Tma$ and $\tmi,\Tmi$ be maximal and minimal relations for system \eqref{4.8} corresponding to the equation \eqref{4.5} (see Proposition \ref{pr4.1}). It follows from Proposition \ref{pr4.1}, (ii) that there exists the limit
\begin{gather*}
[y,z]_b:=\lim_{t\uparrow b}(\cJ\bold y(t), \bold z(t)), \quad y,z\in\dom\sma.
\end{gather*}
This fact enables one to define the linear relation $\cS_a$ in $\lIW$ and the minimal linear relation $\Smi$ in $\LIW$ for the equation \eqref{4.5} by setting
\begin{gather*}
\cS_a=\{\{y(\cd), f(\cd)\}\in\sma: \bold y(a)=0 \;\;{\rm and}\;\; [y,z]_b=0\;\; \text{for every}\;\; z\in\dom\sma\}
\end{gather*}
and $\Smi=\wt\pi_\D \cS_a$. It follows from Proposition \ref{pr4.1} that \begin{gather}\label{4.13}
(\cU_2\oplus \cU_2)\Smi=\Tmi,
\end{gather}
where $\cU_2$ is a unitary operator defined in Proposition \ref{pr4.1}, (iii). This and \eqref{4.12} imply that $\Smi$ is a closed symmetric linear relation in $\LIW$ and $\Smi^*=\Sma$.

For $\l\in\bC$ denote by $\cN_\l$ the linear space of all solutions $y(\cd)$ of  \eqref{4.5} belonging to $\lIW$. The numbers $N_+=\dim \cN_{i}$ and  $N_-=\dim \cN_{-i}$ will be called the formal deficiency indices of the equation \eqref{4.5}. It follows from Proposition \ref{pr4.1}, (i) that $N_\pm$ are formal deficiency indices of the system \eqref{4.8}. Therefore  $N_\pm=\dim \cN_\l,\; \l\in\bC_\pm,$ and  $mr\leq N_\pm\leq 2mr.$
\subsection{Differential equations with matrix-valued coefficients}
Similarly to Hamiltonian systems the equation \eqref{4.5} is called definite if there is only a trivial solution $y=0$ of the equation $l[y]=0$ satisfying $\D(t)y(t)=0$ (a.e. on $\cI$).

Let $\cJ$ be the operator \eqref{4.7.1}. Below we suppose that $U\in \B ((\bC^m)^{2r},(\bC^m)^{r} )$ is an operator satisfying
\begin{gather}\label{4.14}
U\cJ U^*=0 \quad{\rm and} \quad \ran U=(\bC^m)^r.
\end{gather}
\begin{definition}\label{def4.4}
The equation \eqref{4.5} will be called $U$-definite if there exists only the trivial solution $y=0$ of the equation $l[y]=0$ such that $U\bold y(a)=0$ and $\D(t)y(t)=0$ (a.e. on $\cI$).
\end{definition}
It follows from Assertion \ref{ass3.2} and Proposition \ref{pr4.1} that the equality
\begin{gather}\label{4.15}
S=\{\wt\pi_\D \{y,f\}:\{y,f\}\in \sma, \, U\bold y(a)=0\;\; {\rm and}\;\; [y,z]_b=0, \, z\in\dom\sma  \}
\end{gather}
defines a symmetric extension $S$ of $\Smi$ and $S^*=\wt \pi_\D \cS_*$, where $\cS_*$ is a linear relation in $\lIW$ given by $\cS_*=\{\{y,f\}\in\sma:U\bold y(a)=0\}$. Clearly, the domain of $\cS_*$ is
\begin{multline}\label{4.15.0}
\dom\cS_*=\{y\in\dom l\cap \lIW: \, l[y]=\D(t) f_y(t) \\
(\text{a.e. on}\;\; \cI)\;\; \text{with some} \;\; f_y(\cd)\in \lIW
\;\; {\rm and}\;\;U\bold y(a)=0 \}.
\end{multline}
In the following we put $\gH':=\LIW$ and $\gH_0':=\gH'\ominus \mul S$. We will also denote by $\cK',\; \cK_0$ and $\cE$ the linear manifolds in $\lIW$ defined by
\begin{gather}
\cK'=\{f(\cd)\in\lIW: \text{ there exists  a solution} \; y(\cd)\in\dom l\; \text{ of \eqref{4.5} such}\quad\label{4.15.0.1}\\
\qquad\qquad\text{ that}\;\; \D(t) y(t)=0\;\; (\text {a.e. on} \;\; \cI), \; U \bold y(a)=0 \;\;{ \rm and} \;\; [y,z]_b=0,\;z\in\dom\sma\}\nonumber\\
\cK_0=\{f(\cd)\in\lIW: (f(\cd),g(\cd))_\D=0 \;\; \text{for any}\;\; g(\cd)\in\cK'\}.\label{4.15.1}\\
\cE=\{y(\cd)\in \dom S_*:f_y(\cd)\in\cK_0\}\label{4.15.2}
\end{gather}
Clearly, $\mul S=\pi_\D\cK'$  and $\gH_0'=\pi_\D\cK_0$.

Let $\f_U(\cd,\l)(\in B ((\bC^m)^r,\bC^m)$ be the operator solution of \eqref{4.5} such that the corresponding operator-function $\pmb\f_U(t,\l):\cI\to B ((\bC^m)^r,(\bC^m)^{2r})$ given by
\begin{gather} \label{4.16}
\pmb\f_U(t,\l)=(\f_U(t,\l),\,\dots,\,\f_U^{[r-1]}(t,\l),\, \f_U^{[2r-1]}(t,\l),\, \dots,\, \f_U^{[r]}(t,\l) )^\top
\end{gather}
satisfies $\pmb\f_U(a,\l)=-\cJ U^*$.
\begin{definition}\label{def4.5}
A distribution function $\s(\cd):\bR\to \B((\bC^m)^r)$ is called a pseudospectral function of the equation \eqref{4.5}  if:

(i) for each function $f(\cd)\in\lIW$ there exists a function $\wh f(\cd)\in \lSm$ such that
\begin{gather}\label{4.18}
\wh f(s)=\int_\cI \f_U^*(t,s) \D(t)  f(t)\,dt.
\end{gather}
(the integral in \eqref{4.18} converges in $\lSm$, c.f. Definition \ref{def3.3}, {\rm (i)});

(ii) $\pi_\s \wh f(\cd)=0, \; f(\cd)\in\cK',$  and   $||\wh f(\cd)||_{\lSm}=||f(\cd)||_{\lIW},\;f(\cd)\in\cK_0$,
\end{definition}
The operator-function $\wh f(\cd)\in \lSm$ defined  by \eqref{4.18} is called the (generalized) Fourier transform of a function $f(\cd)\in\lIW$. Clearly, the function $\wh f(\cd)$ is defined by $f(\cd)$ uniquely up to the $\s$- equivalence.
\begin{definition}\label{def4.6}
A distribution function $\s(\cd):\bR\to \B((\bC^m)^r)$ is called a spectral function of the equation  \eqref{4.5} if for each function $f(\cd)\in\lIW$ with compact support the  Parseval equality $||\wh f(\cd)||_{\lSm}=||f(\cd)||_{\lIW}$ holds.
\end{definition}
Note that Remark \ref{rem3.3.2} and Definition \ref{def3.3.3} of an orthogonal pseudospectral (spectral) function remain valid, with the obvious modifications, for equation \eqref{4.5}.

By  using Proposition \ref{pr4.1} one can easily prove the following assertion.
\begin{assertion}\label{ass4.7}
A distribution function $\s(\cd):\bR\to \B((\bC^m)^r)$ is a pseudospectral (spectral) function of the system \eqref{4.8} with respect to the Fourier transform
\begin{gather}\label{4.20}
\wh{\bold f}(s)=\int_\cI \pmb\f_U^*(t,s) \wt\D(t)  \bold f(t)\,dt, \quad \bold f(\cd)\in \cL_\D^2(\cI;(\bC^m)^{2r})
\end{gather}
if and only if it is a pseudospectral (resp. spectral) function of the equation \eqref{4.5} with respect to the Fourier transform \eqref{4.18}; moreover, $\wh{\bold y} (s)=\wh y(s), \; y(\cd)\in\dom\sma$.
\end{assertion}
Applying Theorem \ref{th3.12}, Proposition \ref{pr3.13} and Theorems \ref{th3.16}, \ref{th3.16.1} to system \eqref{4.8} and taking  Assertion \ref{ass4.7} into account we arrive at the following theorems.
\begin{theorem}\label{th4.8}
 Assume that:

${\rm (A1')}$ equation  \eqref{4.5} has equal formal deficiency indices $N_+=N_-=:d$;

${\rm (A2')}$  $U\in\B((\bC^m)^{2r},(\bC^m)^r)$ is an operator  satisfying  \eqref{4.14} and equation  \eqref{4.5} is $U$-definite.

Then:  {\rm (i)} there exists a Nevanlinna operator function $M(\cd)$ of the form \eqref{3.33} (with $p=mr$) such that the equalities \eqref{3.36} and \eqref{3.37} establish a bijective correspondence  $\s(\cd)=\s_\tau(\cd)$ between all functions $\tau=\tau(\cd)\in \wt R (\bC^{d-mr})$ satisfying the  condition \eqref{3.37.1} (i.e., all admissible boundary parameters) and  all pseudospectral functions $\s(\cd)$ of the equation \eqref{4.5}.  Moreover, all functions $\tau(\cd)\in \wt R (\bC^{d-mr})$ satisfy \eqref{3.37.1} if and only if $\mul S=\mul S^*$.

{\rm (ii)} The set of spectral functions of the equation  \eqref{4.5} is not empty  if and only if $\cK'\subset \ker \pi_\D$ (or, equivalently, $\mul S=0$). Moreover, in this case the sets of  spectral and pseudospectral functions coincide and hence statement {\rm (i)} holds for spectral functions.
\end{theorem}
\begin{theorem}\label{th4.11}
Let for differential equation \eqref{4.5} the assumptions ${\rm (A1')}$ and ${\rm (A2')}$ in Theorem  \ref{th4.8} and the following assumption ${\rm (A3')}$ be satisfied:

${\rm (A3')}$ $(G_{0b}, G_{1b})^\top:\dom \sma \to \bC^{d-mr}\oplus \bC^{d-mr}$ is a surjective linear operator satisfying
\begin{gather*}
[y,z]_b=(G_{0b}y,G_{1b}z)- (G_{1b}y,G_{0b}z),\quad y,z \in \dom\sma.
\end{gather*}
Assume also that $\cE\subset \dom \cS_*$ is linear manifold  \eqref{4.15.2} and let $\tau (\cd) \in \wt R (\bC^{d-mr})$ be a relation-valued function satisfying  \eqref{3.37.1}, let $\s(\cd)=\s_\tau(\cd)$ be the corresponding  pseudospectral function of the equation  and let  $\eta_\tau\in \C (\bC^{d-mr})$ be the linear relation defined in Theorem \ref{th2.9}. Then for each function $y(\cd)\in\cE$ satisfying the boundary condition $\{G_{0b}y(\cd), -G_{1b}y(\cd)\} \in \eta_\tau$ the following statements hold:

{\rm (i)} If $\wh y(\cd)$ is the Fourier transform \eqref{4.18} of $y(\cd)$, then for each $t\in\cI$
\begin{gather}\label{4.22}
y^{[k]}(t)=\int_{\bR} \f_U^{[k]}(t,s) d\s(s) \wh y(s), \;\; k\in \{0,1, \dots, 2r-1\},
\end{gather}
where the integral exists as the Lebesgue integral (in the same sense as the integral  in \eqref{3.52}).

{\rm (ii)} The integral in \eqref{4.22} converges uniformly on each compact interval $[a,c]\subset \cI$ in the same sense as integral in \eqref{3.52} (see Theorem \ref{th3.16}, {\rm (ii)}).

If in addition $\cK'\subset \ker \pi_\D$ (or, equivalently, $\mul S=0$),  then   $\s(\cd)=\s_\tau (\cd)$ is a spectral function and  statements {\rm (i)} and {\rm (ii)} hold for any function $y(\cd)\in\dom \cS_*$ satisfying the boundary condition   $\{G_{0b}y(\cd), -G_{1b}y(\cd)\} \in \eta_\tau$.
\end{theorem}
\begin{remark}\label{rem4.11.1}
(i) In the case of the regular  equation \eqref{4.5}  one has $d=2mr$. In this case for $y\in \dom\sma$ one can put
\begin{gather*}
G_{0b}y=  y(b) \oplus y^{[1]}(b) \oplus \dots \oplus y^{[r-1]}(b), \quad G_{1b}y= y^{[2r-1]}(b) \oplus  y^{[2r-2]}(b)\oplus \dots\oplus y^{[r]}(b).
\end{gather*}
(ii) If the weight $\D(t)$ is invertible a.e. on $\cI$, then the condition  $\cK'\subset \ker \pi_\D$ in the last statement of Theorem \ref{th4.11} is obviously satisfied.
\end{remark}
\subsection{Scalar differential  equations}
In the case $m=1$ the differential expression $l[y]$  of the form \eqref{4.1} and the  equation \eqref{4.5} will be called a scalar expression and scalar equation respectively. Clearly, in this case the coefficients $p_j(\cd), \; q_j(\cd)$ and the weight  $\D(\cd)$ are real-valued functions.

It is easy to see that for scalar equation \eqref{4.5} the assumption ${\rm(A1')}$ in Theorem \ref{th4.8} is automatically satisfied.
\begin{lemma}\label{lem4.12}
Let $l[y]$ be a scalar expression \eqref{4.1} on an interval $\cI=[a,b\rangle$, let $B\subset \cI$ be a Borel set and let $y(\cd)\in\dom l$ be a function such that $y(t)=0$ (a.e. on $B$). Then $y^{[k]}(t)=0$ (a.e. on $B$), $k\in\{0,1, \dots, 2r\}$, that is there is a Borel set $B_0\subset B$ such that $\mu (B\setminus B_0)=0$, $y^{[2r]}(t)$ exists for each $t\in B_0$ and $y^{[k]}(t)=0, \; t\in B_0, \; k\in\{0,1, \dots, 2r\}$.
\end{lemma}
\begin{proof}
Clearly, it is sufficient to prove the lemma for the case of a compact interval $\cI=[a,b]$. Moreover, we may assume without loss of generality that $y(t)=0, \; t\in B$.

Since $y(\cd)$ is absolutely continuous, there exists a Borel set $B'\subset \cI$ such that $\mu(\cI\setminus B')=0$, the derivative  $y'(t)$ exists for each $t\in B'$ and $y'(\cd)$ is a Borel measurable function on $B'$. Let $B_1:=B'\cap B$. Then $B_1\subset B, \; B_1\in\cA$, $\mu (B\setminus B_1)=0$ and $y'\up B_1$ is a Borel measurable function. Hence for the set $B_{00}':=\{ t\in B_1: y'(t)=0\}$ one has $B_{00}'\subset B_1\subset B$, $B_{00}'\in\cA$ and $y'(t)=0,\; t\in B_{00}'$. Next we show that $\mu (B\setminus B_{00}')=0$.

Denote by $B_2$ the set of all limit points of $B_1$ belonging to $B_1$. Assume that $t\subset B_2$. Then there exists a sequence $\{t_n\}_1^\infty$ such that $t_n\in B_1,\; t_n\neq t$ and $t_n\to t$. Moreover, $t_n,t\in B$ and, consequently, $y(t_n)=y(t)=0$. Note also that $t\in B_1$ and hence there exists the derivative
\begin{gather*}
y'(t)=\lim_{n\to\infty}\frac {y(t_n)-y(t)}{t_n-t}=0.
\end{gather*}
Thus $B_2\subset B_{00}'\subset B_1$ and, consequently, $(B_1\setminus B_{00}') \subset (B_1\setminus B_2)$. Recall that the lower Lebesgue measure $\mu_*(B)$ of the  set $B\subset \cI$ is defined by
\begin{gather*}
\mu_*(B)=\sup\{\mu (F): \, F\subset B \;\; \text{\rm and} \;\; F \;\;\text{\rm is closed} \}
\end{gather*}
and $\mu (B)=\mu_*(B)$ for $B\in \cA$. Since $B_1\setminus B_2$ is the set of all isolated points of $B_1$, it follows that all pints of a closed set $F\subset (B_1\setminus B_2)$ are isolated. Since $F$ is bounded, this implies that $F$ is finite and hence $\mu (F)=0$. Therefore $\mu_* (B_1\setminus B_2)=0$ and the relations
\begin{gather*}
0\leq\mu(B_1\setminus B_{00}')=\mu_*(B_1\setminus B_{00}')\leq \mu_*(B_1\setminus B_2)=0
\end{gather*}
show that $\mu(B_1\setminus B_{00}')=0$. Moreover, $B\setminus B_{00}'=(B_1\setminus B_{00}')\cup (B\setminus B_1)$, which yields the required equality $\mu(B\setminus B_{00}')=0$. Since $y^{[1]}(t)=y'(t)$ (a.e. on $\cI$), this implies that there is a Borel set $B_{00}\subset B$ such that $\mu (B\setminus B_{00})=0$ and $y^{[1]}(t)=0, \; t\in B_{00}$.
Now by using the above method one proves step by step the existence of Borel sets $B_{0k}\subset B$ such that $\mu (B\setminus B_{0k})=0$ and $y^{[k]}(t)=0, \; t\in B_{0k}, \; k\in \{0,1,\, \dots,\, 2r\}$. Finally, letting $B_0=\bigcap\limits_{k=0}^{2r} B_{0k} $ we obtain the set $B_0$ with the required properties.
\end{proof}
As usual we denote by  $\mu (\D>0)$  the Borel  measure of the set
\begin{gather*}
B_+:=\{t\in\cI: \D(t)>0\}.
\end{gather*}
\begin{proposition}\label{pr4.12.1}
For the scalar equation \eqref{4.5} the following statements are equivalent:

{\rm (\romannumeral 1)} The weight function $\D(\cd)$ is nontrivial, that is
\begin{gather}\label{4.24}
\mu (\D>0)\neq 0.
\end{gather}
{\rm (\romannumeral 2)} The equation \eqref{4.5} is definite.

{\rm (\romannumeral 3)} The  equation \eqref{4.5} is $U$-definite for any operator $U\in\B (\bC^{2r},\bC^r)$ satisfying
\begin{gather}\label{4.24.1}
U\cJ U^* =0 \quad {\rm  and} \quad \ran U=\bC^r.
\end{gather}

{\rm (\romannumeral 4)} There exists an operator  $U\in \B (\bC^{2r}, \bC^r)$ such that \eqref{4.24.1} holds and the equation \eqref{4.5} is $U$-definite.
\end{proposition}
\begin{proof}
{\rm (\romannumeral 1) $\Rightarrow$ (\romannumeral 2)}. Assume that a function $y(\cd)\in\dom l$ satisfies $l[y]=0$ and $\D(t) y(t)=0$ (a.e. on $\cI$).  Then $y(t)=0, \; t\in B_+,$ and by Lemma \ref{lem4.12} there is a Borel set $B_0\subset B_+$ such that $\mu (B_+\setminus B_0)=0$ and $y^{[k]}(t)=0,\; t\in B_0, \; k\in\{0,1, \, \dots,\, 2r-1\}$. Since $\mu (B_+)>0$, it follows that $B_0\neq \emptyset$ and hence $y(t)=0, \; t\in\cI$. Thus the equation \eqref{4.5} is definite.

The implications {\rm (\romannumeral 2) $\Rightarrow$ (\romannumeral 3)} and {\rm (\romannumeral 3) $\Rightarrow$ (\romannumeral 4)} are obvious.

{\rm (\romannumeral 4) $\Rightarrow$ (\romannumeral 1)}. If $\mu(\D>0)=0$, then $\D(t)y(t)=0$ (a.e. on $\cI$) for each solution $y(\cd)$ of the equation $l[y]=0$ satisfying $U \bold y(a)=0$ and hence the equation \eqref{4.5} is not $U$-definite. This implies that $\mu (\D>0)\neq 0$.
\end{proof}
\begin{theorem}\label{th4.13}
In the case of a scalar differential equation \eqref{4.5} the corresponding minimal relation $\Smi$ is a densely defined operator  in $\gH'$
\end{theorem}
\begin{proof}
Let for scalar equation \eqref{4.5} $B_0':=\cI\setminus B_+=\{t\in\cI: \D(t)=0\}$. Assume that  $y(\cd)\in\dom l$ and $\D(t)y(t)=0$ (a.e. on $\cI$). Then obviously $y(t)=0$ (a.e. on $B_+$) and by Lemma \ref{lem4.12}  the following statement is valid:

{\rm (S)} If $y(\cd)\in\dom l$ and $\D(t)y(t)=0$ (a.e. on $\cI$), then  $l[y]=0$ (a.e. on $B_+$).

Let $\cL''$ be the set of all functions $f(\cd)\in\cL_\D^2(\cI;\bC)$ such that there exists a solution $y(\cd)\in \dom l$ of \eqref{4.6} satisfying $\D(t)y(t)=0$ (a.e. on  $\cI$). In view of statement (S) for each $f(\cd)\in\cL''$ one has $\D(t)f(t)=0$ (a.e. on $B_+$). This and the equality $\D(t)f(t)=0, \; t\in B_0',$ imply that $\D(t)f(t)=0$ (a.e. on $\cI$) and hence
\begin{gather}\label{4.25}
\pi_\D f(\cd)=0, \quad f(\cd)\in \cL''.
\end{gather}
Since obviously $\mul\Sma=\pi_\D \cL''$, it follows from \eqref{4.25} that $\mul\Sma=\{0\}$. This yields the required statement.
\end{proof}
\begin{theorem}\label{th4.14}
Let for   scalar equation \eqref{4.5}  the weight function $\D(t)$  satisfies \eqref{4.24} and  let $U\in \B(\bC^{2r},\bC^r)$ be an operator satisfying \eqref{4.24.1}.  Then the set of spectral functions $\s(s)(\in \B(\bC^r))$ of this equation (with respect to the Fourier transform \eqref{4.18}) is not empty and  there exists a Nevanlinna operator-function \eqref{3.33} (with $p=r$) such that the equalities \eqref{3.36} and \eqref{3.37} give a bijective correspondence $\s(\cd)=\s_\tau (\cd)$ between all (arbitrary) functions $\tau=\tau(\cd)\in \wt R(\bC^{d-r})$ and all spectral functions $\s(\cd)$ of \eqref{4.5}. Moreover, a spectral function $\s_\tau(\cd)$ is orthogonal if an only if $\tau(\l)\equiv\t(=\t^*), \; \l\in\CR$.
\end{theorem}
\begin{proof}
First observe that by Proposition \ref{pr4.12.1} the equation \eqref{4.5} is $U$-definite and hence  the assumptions (${\rm A1^\prime}$) and (${\rm A2'}$) in Theorem \ref{th4.8} are satisfied. Next, the relation $S$ (see \eqref{4.15})  is a symmetric extension of $\Smi$ and  by Theorem \ref{th4.13} $\Smi$ is a densely defined operator.  Therefore  $S$ is a densely defined operator as well and  hence
\begin{gather}\label{4.29}
\mul S=\mul S^*=\{0\}.
\end{gather}
Now the required statement follows from Theorem \ref{th4.8}.
\end{proof}
In the following theorem we provide sufficient conditions  for the uniform convergence of integrals in \eqref{4.22} with a spectral function $\s(\cd)$ of the scalar equation.
\begin{theorem}\label{th4.15}
Let for scalar differential equation \eqref{4.5} the assumptions of Theorem \ref{th4.14} be satisfied and let the assumption ${\rm (A3')}$ in Theorem \ref{th4.11} be fulfilled.  Moreover, let $\tau=\tau(\cd)\in \wt R(\bC^{d-r})$, let $\s(\cd)=\s_\tau(\cd)$ be the corresponding spectral function of \eqref{4.5} (see Theorem \ref{th4.14}) and let $\eta_\tau\in \C (\bC^{d-r})$ be the linear relation defined in Theorem \ref{th2.9}. Denote by $\cF$ the set of all functions $y(\cd)\in\dom l\cap \cL_\D^2(\cI;\bC)$ satisfying the equality $l[y]=\D(t) f_y(t)$ (with some $f_y(\cd)\in\cL_\D^2(\cI;\bC)$) and the  boundary conditions
\begin{gather}\label{4.30}
U\bold y(a)=0, \qquad \{G_{0b}y(\cd),- G_{1b}y(\cd) \}\in\eta_\tau.
\end{gather}
Then for each function $y(\cd)\in\cF$ statements {\rm (i)} and  {\rm (ii)} of Theorem \ref{th4.11} hold.
\end{theorem}
\begin{proof}
First observe that by Proposition \ref{pr4.12.1} the equation \eqref{4.5} is $U$-definite and hence the assumptions (${\rm A1'}$) -- (${\rm A3'}$) in Theorems \ref{th4.8} and \ref{th4.11} are satisfied.

Assume that $y(\cd)\in\cF$. Then $y(\cd)\in\dom S_*$ (see \eqref{4.15.0}) and $\{G_{0b}y(\cd),- G_{1b}y(\cd) \}\in\eta_\tau$. Moreover, by \eqref{4.29} $\mul S=\{0\}$. This and the last statement in Theorem \ref{th4.11} yield the required statement.
\end{proof}
Next consider scalar regular equation \eqref{4.5} on an interval $\cI=[a,b]$ (see Definition \ref{def4.0}).  Clearly for such equation one has $d(=N_\pm)=2r$.

Let $U\in\B (\bC^{2r},\bC^r)$ be an operator satisfying \eqref{4.24.1}. Then there exists an operator $U'\in\B (\bC^{2r},\bC^r)$  such that the operator $\wt U=(U', U)^\top \in \B (\bC^{2r})$ satisfies $\wt U^* \cJ \wt U =\cJ$. Let as before $\f_U(\cd,\l)(\in \B (\bC^r,\bC))$ be an operator solution of \eqref{4.5} satisfying $\pmb\f_U(a,\l)=-\cJ U^*$ and let $\psi (\cd, \l)$ be similar solution with $\pmb\psi(a,\l)=\cJ (U')^*$. Clearly, $\f_U(\cd,\l)$ and $\psi (\cd, \l)$ are components of the solution $Y(t,\l)= (\f_U(t,\l), \,\psi(t,\l) )(\in\B(\bC^r\oplus \bC^r, \bC))$ of \eqref{4.5} satisfying $\wt U \bold Y(a,\l)=I_{2r}$.

Below with  a function $\tau(\cd)\in \wt R(\bC^r)$ represented in the ''canonical'' form \eqref{2.9} we associate a pair of operator functions $C_{j\tau}(\cd):\CR\to \B (\bC^r), \; j\in \{0,1\},$ given by
\begin{gather}\label{4.33.1}
C_{0\tau}(\l)={\rm diag}\, (-\tau_0(\l), I_\cK), \qquad C_{1\tau}(\l)={\rm diag}\, (I_{\cH_0}, 0), \quad \l\in\CR.
\end{gather}
It is easy to see that
\begin{gather*}
\tau(\l)=\{\{h,h'\}\in\bC^r\oplus \bC^r: C_{0\tau}(\l)h+C_{1\tau}(\l)h'=0 \}, \quad \l\in\CR.
\end{gather*}
In the case of a regular equation \eqref{4.5} Theorem \ref{th4.14} can be reformulated in the form of the following theorem.
\begin{theorem}\label{th4.18.2}
Let for regular scalar equation \eqref{4.5} the assumptions of Theorem \ref{th4.14} be satisfied and let $w_j(\l)(\in \B(\bC^r))$ be the operator functions given by
\begin{gather}
w_1(\l)= (\f_U(b,\l), \,\f_U^{[1]}(b,\l),\, \dots,\,  \f_U^{[r-1]}(b,\l) )^\top  \label{4.33.2}\\
w_2(\l)= (\psi(b,\l), \,\psi^{[1]}(b,\l),\, \dots,\,  \psi^{[r-1]}(b,\l) )^\top  \\
w_3(\l)= (\f_U^{[2r-1]}(b,\l), \,\f_U^{[2r-2]}(b,\l), \,\dots, \, \f_U^{[r]}(b,\l) )^\top  \\
w_4(\l)= (\psi^{[2r-1]}(b,\l), \,\psi^{[2r-2]}(b,\l), \,\dots, \,  \psi^{[r]}(b,\l) )^\top . \label{4.33.3}
\end{gather}
 Then the equality
\begin{gather}\label{4.33.5}
m_\tau(\l)=(C_{0\tau}(\l)w_1(\l)+C_{1\tau}(\l)w_3(\l))^{-1}(C_{0\tau} (\l)w_2(\l)+C_{1\tau}(\l)w_4(\l))
\end{gather}
together with \eqref{3.37} gives a bijective correspondence $\s(\cd)=\s_\tau (\cd)$ between all  functions $\tau=\tau(\cd)\in \wt R(\bC^{r})$ and all spectral functions $\s(\cd)$ of \eqref{4.5} (with respect to the Fourier transform \eqref{4.18}).
\end{theorem}
\begin{proof}
Consider the Hamiltonian system \eqref{4.8} corresponding to the equation \eqref{4.5} (see Proposition \ref{pr4.1}). Let $T$ be  symmetric relation \eqref{3.22.1} for system \eqref{4.8} and let $S$ be symmetric relation \eqref{4.15} for equation \eqref{4.5}. Then by \eqref{4.29} and Proposition \ref{pr4.1}  $\mul T=\mul T^*=\{0\}$ and  by Theorem \ref{th3.12} and Proposition \ref{pr3.13} the equalities \eqref{3.36} and \eqref{3.37} give a parametrization of all spectral functions $\s(\cd)$ of \eqref{4.8} in terms of functions $\tau(\cd)\in \wt R (\bC^r)$.

Let $\bm\f_U(t,\l)=(\bm\f_{0U}(t,\l),\, \bm\f_{1U}(t,\l))^\top$ and $\bm\psi(t,\l)=(\bm\psi_{0}(t,\l),\, \bm\psi_{1}(t,\l))^\top$ be $\B (\bC^r,\bC^r\oplus \bC^r)$-valued operator solutions of \eqref{4.8} with the initial values $\bm\f_U(a,\l)=-\cJ U^* $ and $\bm\psi(a,\l)=\cJ (U')^* $. Then according to \cite{Sah13,Mog15} the equality \eqref{3.36} can be written in the form \eqref{4.33.5} with $w_1(\l)=\bm\f_{0U}(b,\l)$, $w_2(\l)=\bm\psi_{0}(b,\l)$, $w_3(\l)=\bm\f_{1U}(b,\l)$ and $w_4(\l)=\bm\psi_{1}(b,\l)$. Moreover, by Proposition \ref{pr4.1}, (i) $w_j(\l)$ admit the representation \eqref{4.33.2} -- \eqref{4.33.3} and Assertion \ref{ass4.7} yields the required statement.
\end{proof}
\subsection{Scalar Sturm - Liouville equations}\label{sub5.4}
The results of this section take an  especially simple form in the case $m=1$ and  $r=1$, i.e., in the case of the scalar Sturm -Liouville equation \eqref{1.7}. Below we give the proof of Theorem \ref{th1.2} concerning this equation.
\begin{proof}
(i) It is clear that the operators $U = (-\cos a, -\sin \a)$ and $U'=(-\sin \a, \cos a)$ satisfy the assumptions before Theorem \ref{th4.18.2} and the corresponding solutions  $\f_u(\cd,\l)=\f(\cd,\l)$ and $\psi (\cd,\l)$  of \eqref{1.7} are defined by initial values specified in the theorem. This and Theorem \ref{th4.18.2} give statement (i).

(ii) In view of \eqref{2.13} the linear relation $\eta_\tau$ in $\bC$ is defined as follows:

(1) if $\lim\limits_{y\to\infty} \frac {\tau(iy)} {iy}\neq 0$, then $\eta_\tau=\{0\}\oplus\bC$;

(2) if \eqref{1.19} holds, then $\eta_\tau=h\oplus(-D_\tau h), \; h\in\bC,$ with $D_\tau=\lim\limits_{y\to\infty}\tau(iy)$;

(3) if $\lim\limits_{y\to\infty} \frac {\tau(iy)} {iy}= 0$ and $ \lim\limits_{y\to\infty} y\im \tau(iy)=\infty$, then $\eta_\tau=\{0\}$.

\noindent Note also that according to Remark \ref{rem4.11.1} one can put in \eqref{4.30} $G_{0b} y=y(b)$ and $G_{1b} y=y^{[1]}(b) $. Now statement (ii) follows from Theorem \ref{th4.15}.
\end{proof}
For given $\a,\b\in\bR$ consider the eigenvalue problem \eqref{1.7}, \eqref{1.13} (cf.  Theorem \ref{th1.1}).   We assume that $p,q$ and $\D$ in \eqref{1.7}  are real-valued functions on a compact interval $I=[a,b]$ such that $\tfrac 1 p, q$ and $\D$ are integrable on $\cI$ and $\D(t)\geq 0, \; t\in\cI$ (we do not assume that $\D(t)>0,\; t\in\cI$). A function $y\in \dom l$ is called a solution of the problem \eqref{1.7}, \eqref{1.13} if $l[y]=\l \D(t)y$ (a.e. on $\cI$) and \eqref{1.13} is satisfied. The set of all solutions of this problem will be denoted by $L_\l$ (it is clear that $L_\l$ is a finite-dimensional subspace in $\cL_\D^2(\cI;\bC)$). Denote also  by $EV$ the set of all eigenvalues of the problem \eqref{1.7}, \eqref{1.13}, i.e., the set of all $\l\in\bC$ such that $L_\l\neq\{0\}$. For each $\l\in EV$ the subspace $L_\l\subset \cL_\D^2(\cI;\bC)$ is called an eigenspace and a function $y\in L_\l$ is called an eigenfunction.
\begin{corollary}\label{cor4.18.5}
Let the weight function $\D(\cd)$ in \eqref{1.7} satisfies $\mu(\D>0)\neq 0$. Then:

{\rm (i)} $EV$ is an infinite countable subset  in $\bR$ without finite limit points and $\dim L_\l=1, \; \l\in EV$.

{\rm (ii)} If in addition $p(t)\geq 0, \,t\in\cI, $ then the set  $EV$ has properties from statement {\rm (i)}  and, moreover, it  is bounded from below (the latter means that  there exists $\l_0\in EV$ such that $\l_0\leq \l, \; \l\in EV$).

{\rm (iii)} Let $\{\l_k\}_1^\infty$ be a sequence of all   eigenvalues $\l_k\in EV$ and let $v_k\in L_{\l_k}$ be an eigenfunction with $ ||v_k||_{\cL_\D^2(\cI;\bC)}=1, \; k\in \bN$. Denote by $\cF'$ the set of all functions $y\in\dom l $ such that $l[y]= \D f_y$ (a.e. on $\cI$) with some $f_y\in\cL_\D^2(\cI;\bC)$ and the boundary conditions \eqref{1.13} are satisfied. Then each function $y\in\cF'$ admits an eigenfunction expansion \eqref{1.14}, which converges  absolutely and uniformly on $\cI$.
\end{corollary}
\begin{proof}
First we give the proof for the case $\sin\b\neq 0$. In this case \eqref{1.13} is equivalent to
\begin{gather}\label{4.33.17}
\cos\a\cd y(a) + \sin \a \cd y^{[1]}(a)=0, \qquad y^{[1]}(b)=\t y(b),
\end{gather}
where $y^{[1]}(t)$ is the same as in Theorem \ref{th1.2} and $\t=- {\rm ctg}\, \b$.

(i)  Let $U=(-\cos\a, -\sin\a)$, let $\f(\cd,\l)$ and $\psi(\cd,\l)$ be solutions of \eqref{1.7} from Theorem \ref{th1.2}  and let $\tau\in R[\bC]$ be given by $\tau (\l)\equiv \t(=\ov \t), \, \l\in\bC$. Then $\f(\cd,\l)=\f_U(\cd,\l)$ and by Theorem \ref{th1.2}, (i) the equality \eqref{1.10} with $\tau(\l)\equiv \t$  defines a function $m(\cd)=m_\tau(\cd)\in R[\bC]$ such that formula \eqref{1.11} gives a spectral function $\s (\cd)=\s_\tau(\cd)$ of the equation \eqref{1.7}. Since the function $m(\cd)$ is a quotient of two entire functions, it follows that $m(\cd)$ is a meromorphic function with the finite or countable set $\cP=\{\l_k\}_1^n\; (n\leq\infty)$ of poles, which lies in $\bR$ and has no finite limit points. Hence $\s(\cd)$ is a jump function with jumps $\s_k>0$ at points $\l_k\in\cP$.

Next assume that $S$ is a symmetric relation \eqref{4.15}. Then by \eqref{4.29} $S$ is a densely defined operator in $L_\D^2(\cI;\bC)$. Put
\begin{multline*}
\cL_*=\{y\in\dom l: \cos\a\cd y(a) + \sin \a \cd y^{[1]}(a)=0 \\
\text{ and $\;l[y]= \D f_y$ (a.e. on $\cI$) with some $f_y\in\cL_\D^2(\cI;\bC)$}\}
\end{multline*}
Then the adjoint $S^*$ of $S$ is given by
\begin{gather*}
\dom S^*=\{\pi_\D y:\, y\in \cL_*\}, \qquad S^*(\pi_\D y) =\pi_\D f_y, \; \; y\in\cL_*.
\end{gather*}
It follows from Proposition \ref{pr4.12.1}  that  equation \eqref{1.7} is $U$-definite. Therefore
\begin{gather}\label{4.33.20}
\ker (\pi_\D\up \cL_*)=\{0\}
\end{gather}
and combining of Proposition \ref{pr4.1} with Proposition \ref{pr3.9} and Remark \ref{rem4.11.1} implies that the equalities $\G_0(\pi_\D y)=y(b), \; \G_1(\pi_\D y)=-y^{[1]}(b) , \; y\in \cL_*,$
define a boundary triplet $\Pi=\{\bC,\G_0,\G_1\}$ for $S^*$. Let $\wt S_\tau$ be a self-adjoint extension of $S$ corresponding to $\tau (\l)\equiv \t$ (in the triplet $\Pi$) and let
\begin{gather}\label{4.33.21}
\cL_\tau=\{y\in\cL_*: y^{[1]}(b)=\t y(b)\}.
\end{gather}
Then by Theorem \ref{th2.6}, (ii) $\wt S_\tau$ is an operator in $L_\D^2(\cI;\bC)$ given by
\begin{gather}\label{4.33.22}
\dom \wt S_\tau=\{\pi_\D y:\, y\in \cL_\tau\}, \qquad \wt S_\tau (\pi_\D y) =\pi_\D f_y, \; \; y\in\cL_\tau.
\end{gather}
In the following we denote by $\Si(\wt S_\tau)$  spectrum of $\wt S_\tau$.

According to \cite{Mog15} the Fourier transform \eqref{4.18} defines a unitary operator $V_\s(\pi_\D y)= \wh y,\; y\in \cL_\D^2(\cI;\bC),$ acting  from $L_\D^2(\cI;\bC)$ onto $L_2(\s;\bC)$; moreover,
\begin{gather}\label{4.33.25}
V_\s^*  g=\pi_\D \left( \int_{\bR} \f(\cd,s) g(s)\, d\s(s)  \right), \quad g\in L_2(\s;\bC)
\end{gather}
and the operator $\wt S_\tau$ is unitarily equivalent to the multiplication operator $\L_\s$ in  $L_2(\s;\bC)$ by means of $V_\s$. Therefore $\Si (\wt S_\tau)=\cP=\{\l_k\}_1^n,\; n\leq\infty$, which implies that $\Si (\wt S_\tau)$ coincides with the set of all eigenvalues $\l_k$ of $\wt S_\tau$ and $\dim \ker (\wt S_\tau-\l_k)=1, \; \l_k\in \Si (\wt S_\tau)$. Moreover, it follows from \eqref{4.24} that $\dim L_\D^2(\cI;\bC)=\infty$ and hence the set $\Si (\wt S_\tau)$ is infinite (that is $n=\infty$).  Next, in view of \eqref{4.33.22} and \eqref{4.33.21} $\ker (\wt S_\tau-\l)=\pi_\D L_\l, \; \l\in\bC,$ and \eqref{4.33.20} implies that $\ker (\pi_\D\up L_\l)=\{0\}$. Hence $EV=\Si (\wt S_\tau)$ and $\dim L_\l=\dim\ker (\wt S_\tau -\l)=1, \; \l\in EV$. This proves statement (i).

Statement (ii) can be proved in the same way as Theorem 5 in \cite[\S 19]{Nai}.

(iii) Let $y\in\cF'$, so that \eqref{4.33.17} is satisfied with $\t=\ov\t$. Let as before $\tau(\cd)\in\ R[\bC]$ be given by $\tau(\l)\equiv \t$. Then \eqref{1.19} is satisfied, $D_\tau=\t$ and hence $y$ satisfies boundary conditions (bc2) in Theorem \ref{th1.2}, (ii). Let $\mathcal V_k(t)=\wh y(\l_k)\s_k \f(t,\l_k)$. Then by Theorem \ref{th1.2}, (ii)
\begin{gather*}
y(t)=\int_\bR \f(t,s)\wh y(s) \, d\s(s) =\sum_{k=1}^\infty \mathcal V_k(t),
\end{gather*}
where the series converges absolutely and uniformly on $\cI$. Now it remains to show that $\mathcal V_k\in L_{\l_k}$.

Since $\wt S_\tau$ and $\L_\s$ are unitarily equivalent by means of $V_\s$, it follows that $V_\s^*\, \dom \L_\s =\dom \wt S_\tau$. Moreover, $\wh y(\l_k)\chi_{\{\l_k\}}(\cd)\in  \dom \L_\s$ and by \eqref{4.33.25} $V_\s^* (\wh y(\l_k)\chi_{\{\l_k\}}(\cd))=\pi_\D \mathcal V_k$. Hence $\pi_\D \mathcal V_k \in \dom \wt S_\tau$ and by \eqref{4.33.22} $\pi_\D \mathcal V_k =\pi_\D y$ with some $y\in\cL_\tau (\subset \cL_*)$. On the other hand $\mathcal V_k \in \cL_*$ and \eqref{4.33.20} implies that $\mathcal V_k=y$. Thus $\mathcal V_k\in \cL_\tau$ and, consequently, $\mathcal V_k\in L_{\l_k}$.

In the case $\sin \b =0$ one proves the required statements in the same way by setting $\tau(\l)\equiv \{0\}\oplus\bC, \; \l\in\bC$.
\end{proof}
\begin{remark}\label{rem4.18.6}
Statement (ii) of Corollary \ref{cor4.18.5} was proved by other methods in \cite{EKZ83}. \end{remark}
\subsection{Example} Consider the scalar regular Sturm -Liouville  equation
\begin{gather}\label{4.35}
-y''=\l y,\qquad t\in \cI= [0,1], \quad \l\in\bC,
\end{gather}
on an interval $\cI=[0,1]$. Let
\begin{gather*}
\f(t,\l)=\cos (\sqrt \l \,t), \qquad \psi (t,\l)=\tfrac 1 {\sqrt \l} \sin (\sqrt \l\, t).
\end{gather*}
The immediate checking shows that $\f(\cd,\l)$ and $\psi (\cd,\l)$ are solutions of \eqref{4.35} with   $\f(0,\l)=1, \; \f '(0,\l)=0$ and $\psi (0,\l)=0, \; \psi'(0,\l)=1$. Hence  $\f(\cd,\l)$ and $\psi (\cd,\l)$ satisfy \eqref{1.8} with $\a=-\tfrac \pi 2$ and
\begin{gather*}
\f(1,\l)=\cos \sqrt \l, \quad \f'(1,\l)=-\sqrt \l \sin\sqrt\l, \quad \psi(1,\l)=\tfrac {\sin \sqrt \l} {\sqrt \l},\quad \psi'(1,\l)=\cos \sqrt \l.
\end{gather*}
Therefore by Theorem  \ref{th1.2}, (i) the equality
\begin{gather}\label{4.37}
m_\tau(\l)=(\tfrac {\sin \sqrt \l} {\sqrt \l}\cd\tau (\l)-\cos \sqrt \l)(\cos \sqrt \l\cd\tau (\l)+\sqrt \l \sin\sqrt\l)^{-1}.
\end{gather}
together with \eqref{1.11} describes in terms of the parameter $\tau\in \wh R[\bC]$ all spectral functions of the equation \eqref{4.35} with respect to the Fourier transform
\begin{gather}\label{4.39}
\wh y(s)=\int_{[0,1]} \cos (\sqrt s \,t)y(t)\, dt, \quad y(\cd)\in L^2 [0,1],\;\; s\in\bR.
\end{gather}
Let $\tau=\tau(\l)=\sqrt \l$ and let $\s(\cd)=\s_\tau(\cd)$ be the corresponding spectral function of \eqref{4.35}. Then by \eqref{4.37}
\begin{gather}\label{4.40}
m_\tau(\l)=\frac{\sin \sqrt \l-\cos \sqrt \l}{\sqrt\l(\cos \sqrt \l+\sin\sqrt\l)}, \quad \l\in\CR
\end{gather}
and \eqref{1.11} implies that $\s(\cd)\in AC((-\infty,0);\bR)$  and
\begin{gather}\label{4.41}
\s'(s)=\frac 1 \pi \im m_\tau(s)=\frac 2 {\pi \sqrt {-s}(e^{2\sqrt {-s}}+ e^{-2\sqrt {-s}} )}, \quad s\in (-\infty, 0).
\end{gather}
Moreover,  $m_\tau(\cd)$ is meromorphic on $\bC\setminus (-\infty,0)$ with poles $a_k\in (0,\infty)$ given by
\begin{gather}\label{4.42}
a_k=\pi^2 (k- \tfrac 1 4)^2,\quad k\in\bN .
\end{gather}
Hence $\s(s)$ is  constant on intervals $(0,a_1)$ and $(a_k,a_{k+1}), \; k\in\bN,$ with jumps $\s_k$ in $a_k$ given by
\begin{gather}\label{4.42.1}
\s_k=-\frac{(\sin \sqrt s-\cos \sqrt s)_{s=a_k}}{(\sqrt s(\cos \sqrt s+\sin\sqrt s))_{s=a_k}'}=- \frac{\sin \sqrt a_k-\cos \sqrt a_k}{\frac 1 2 (\cos \sqrt a_k-\sin \sqrt a_k)}=2.
\end{gather}
Note also that by \eqref{4.39}
\begin{gather}
\wh y(s)=\tfrac 1 2 \int_{[0,1]}\left(e^{\sqrt {-s}\,t}+e^{-\sqrt {-s}\,t}\right)y(t)\, dt, \quad s\in (-\infty,0)\label{4.43}\\
\wh y(a_k) =\int_{[0,1]} \cos (\pi (k- \tfrac 1 4)t)y(t)dt, \quad k\in\bN.\label{4.44}
\end{gather}
Now we are ready to prove the following assertion.
\begin{assertion}\label{ass4.18}
Let $y$ be a complex-valued function on $\cI=[0,1]$ such that  $y'$ is absolutely continuous on $\cI$, $y''\in L^2 (\cI)$ and $y'(0)=0, \; y(1)=y'(1)=0$. Then the function $y$ admits the representations
\begin{gather}
y(t)=\frac 1 \pi \int\limits_{(-\infty,0)} \frac {e^{\sqrt {-s}\,t}+e^{-\sqrt {-s}\,t}}{\sqrt{-s}(e^{2\sqrt {-s}}+ e^{-2\sqrt {-s}})} \wh y(s)\, ds +2\sum_{k=1}^\infty \a_k \cos (\pi (k- \tfrac 1 4) t)\label{4.45},%\\
%y'(t)=\frac 1 \pi \int\limits_{(-\infty,0)} \frac {e^{\sqrt {-s}\,t}+e^{-\sqrt {-s}\,t}}{e^{2\sqrt {-s}}+ e^{-2\sqrt {-s}}} \wh y(s)\, ds -2\pi\sum_{k=1}^\infty \a_k (\tfrac 3 4 +k-1)\sin (\pi (\tfrac 3 4 +k-1) t)\label{4.46}
\end{gather}
where $\wh y(s) $ is given by \eqref{4.43} and
\begin{gather}\label{4.47}
\a_k=\int_{[0,1]} y(t) \cos (\pi (k- \tfrac 1 4) t), \quad k\in\bN.
\end{gather}
The integral and series in \eqref{4.45} converge absolutely for each $t\in \cI$ and uniformly on $\cI$.
\end{assertion}
\begin{proof}
Let a function $y(\cd)$ satisfies the assumption of the assertion. Since $\lim\limits_{y\to +\infty} \tfrac {\tau (iy)}{iy} =0$ and $\lim\limits_{y\to +\infty}  y\cd \im  \tau (iy)=\infty $, it follows that $y$ belongs to the set $\cF$ from Theorem \ref{th1.2}. Moreover,  the equality \eqref{1.22} takes the form
\begin{gather}\label{4.48}
y(t)=\int_{(-\infty,0)} \f_U(t,s) \s'(s)\wh y(s)\, ds +\sum_{k=1}^\infty \f_U (t,a_k)\s_k \a_k,
\end{gather}
where $\wh y(s)$ and $\a_k=\wh y(a_k)$ are given by \eqref{4.43} and \eqref{4.47}, $\s'(s)$ is given by \eqref{4.41},
\begin{gather*}
\f_U(t,s)=\cos (i\sqrt{-s}\,t)=\tfrac 1 2 \left(e^{\sqrt {-s}\,t}+e^{-\sqrt {-s}\,t}\right), \quad s\in (-\infty, 0), \;\; t\in [0,1]\\
\f_U(t,a_k)=\cos (\sqrt {a_k}\,t)=\cos (\pi (k- \tfrac 1 4) t)
\end{gather*}
and in view of  \eqref{4.42.1} $\s_k=2$. Now the required statement follows from Theorem \ref{th1.2}.
\end{proof}

\end{document}